\documentclass[3p,preprint,12pt]{elsarticle}

\usepackage{tabulary,xcolor}
\usepackage{amsfonts,amsmath,amssymb}
\usepackage[T1]{fontenc}
\usepackage{floatrow}
\usepackage{pgfplots,algorithmic,algorithm}
\usepackage{booktabs,appendix}
\usepackage{graphicx,amssymb,amsmath,amsthm}
\newcommand{\bx}{\mathbf{x}}
\newcommand{\bX}{\mathbf{X}}
\newcommand{\by}{\mathbf{y}}
\newcommand{\bc}{\mathbf{c}}

\renewcommand{\bf}{\mathbf{f}}
\newcommand{\bu}{\mathbf{u}}

\usepackage{booktabs}

\usepackage{bm}
\usepackage{float}
\journal{Computer Methods in Applied Mechanics and Engineering}
\pgfplotsset{compat=newest}

\usepackage{hyperref}
\usepackage{cleveref}
\Crefname{table}{Table}{Tables}
\Crefname{equation}{Eq.}{Eqns.}

\usepackage{booktabs}
\usepackage{subcaption}
\newtheorem{theorem}{Theorem}
\newtheorem{lemma}{Lemma}
\begin{document}

\begin{frontmatter}

\title{Isogeometric Collocation Method for the Fractional Laplacian in the 2D Bounded Domain}

\author[label1]{Kailai Xu}
\ead{kailaix@stanford.edu}
\author[label1,label2]{Eric Darve}
\ead{darve@stanford.edu}
\address[label1]{Institute for Computational and Mathematical Engineering, Stanford University, Stanford, CA, 94305}
\address[label2]{Mechanical Engineering, Stanford University, Stanford, CA, 94305}
 
\begin{abstract}
We consider the isogeometric analysis for fractional PDEs involving the fractional Laplacian in two dimensions. An isogeometric collocation method is developed to discretize the fractional Laplacian and applied to the fractional Poisson problem and the time-dependent fractional porous media equation. Numerical studies exhibit monotonous convergence with a rate of $\mathcal{O}(N^{-1})$, where $N$ is the degrees of freedom. A comparison with finite element analysis shows that the method enjoys higher accuracy per degree of freedom and has better convergence rate. We demonstrate that isogeometric analysis  offers a novel and promising computational tool for nonlocal problems.  
\end{abstract}

\end{frontmatter}

%\linenumbers
\section{Introduction}

We consider the two-dimensional nonlocal model involving the Riesz fractional Laplacian, which can be defined through the Fourier transform~\cite{kwasnicki2017ten}
\begin{equation}
    \mathcal{F}((-\Delta)^su)(\bm{\xi}) = |\bm{\xi}|^{2s}\mathcal{F}(u)(\bm\xi)
\end{equation}
where $s\in(0,1)$, $\mathcal{F}$ is the Fourier transform, and $u$ is sufficiently smooth. Equivalently, we have
\begin{equation}\label{equ:def}
    (-\Delta)^s u(\bx) = c_{s,2} \; \mathrm{p.v.}\int_{\mathbb{R}^2}\frac{u(\bx)-u(\by)}{|\bx-\by|^{2+2s}}d\by
\end{equation}
where p.v.\ denotes principal integration and $c_{s,2}$ is the normalization constant
\begin{equation}
    c_{s,2} = \frac{2^{2s}\Gamma(1+s)}{\pi|\Gamma(-s)|}
\end{equation}
The nonlocal models naturally arise from many physical processes that exhibit super-diffusive jump phenomena. However, the computational methods for nonlocal models lag drastically behind their integer-order counterparts and are generally more expensive in both storage and computational cost. Thus, the integer-order differential equations are usually used as approximations in physical modeling. Recently with more computational resources available, interests are shifting back to nonlocal models. Many applications can be modeled by partial differential equations involving the fractional Laplacian, such as the reverse time migration~(RTM) for attenuating media~\cite{wang2017adaptive}, the anomalous diffusion in porous media~\cite{de2010fractional}, and image denoising~\cite{gatto2015numerical}. For a more comprehensive treatment involving the theory and numerical methods of fractional Laplacian, we refer readers to the review paper~\cite{lischke2018fractional}.
 
It is worth mentioning that many other formulations of fractional operators for nonlocal models exist \cite{heydari2018legendre,heydari2013two}. For example, the time-fractional models \cite{heydari2019wavelet,hooshmandasl2016numerical} and the space-time fractional models \cite{heydari2019computational} are two alternative fractional operators to the fractional Laplacian operator. The differential operators in the aforementioned fractional operators are usually decoupled in time or space. In contrast, 
the two spatial dimensions of the fractional Laplacian operator \Cref{equ:def} cannot be decoupled and must be considered altogether in the numerical discretization, which poses unique challenges.

The challenges in solving the fractional partial differential equations involving the fractional Laplacian is many folds. Firstly, the kernel \Cref{equ:def} is singular and therefore requires techniques like computing singular integrals in the boundary element method. The corresponding implementation can be quite challenging, especially in the case where the computational domain is complex and high dimensional. For example, Acosta et al.~\cite{acosta2017short} developed a nontrivial finite element method by applying the Duffy-type transforms and  different quadrature rules according to the relative position of two triangle elements~(i.e., identical, sharing a vertex, sharing an edge and disjoint). Secondly, with homogeneous boundary conditions, the solutions to the fractional Poisson equation are continuous but not continuously differentiable across the boundary\cite{ros2014dirichlet}. This reduces the convergence rates of many numerical methods. Thirdly, the memory and computational cost of the fractional partial differential equations is in general higher than its integer-order counterpart since the resultant linear system is usually dense. In sum, for the fractional partial differential equations involving the fractional Laplacian to be practical, designing an easy-to-implement, accurate and efficient numerical method is desirable. 

Many existing methods suffer from their applicability or numerical issues. For example, \cite{xu2018radial} proposed a radial basis function collocation method for decoupled fractional Laplacian wave equations. The resultant linear system is usually ill-conditioned due to its meshless nature. \cite{kyprianou2017unbiased} proposed a ``walk-on-spheres'' Monte Carlo methods for the fractional Laplacian. The method requires a large number of simulations, especially for small $s$, where the ``jump size'' is very large and thus the variance is large. \cite{xu2018spectral} proposed a spectral method for 2D and 3D but the method is only limited to the unit ball domains.

Since the seminal paper by Hughes~et~al.~\cite{hughes2005isogeometric}, there has been extensive research on the isogeometric analysis. The basic idea is to use the same NURBS-based representation for both the solution and the geometry. By doing so, the isogeometric analysis can represent many complex geometries exactly. Additionally, the parametrized solution has high-order continuous derivatives depending on the choice of the NURBS basis functions, which enables us to work directly with the strong form of the fractional Laplacian.

This inspired us to attack the computational challenges in nonlocal models involving the fractional Laplacian with the isogeometric collocation method. Compared to the finite element analysis~\cite{acosta2017short,ainsworth2017aspects}, the isogeometric collocation method is much easier to implement and we show that our method is more accurate per degree of freedom than the finite element analysis due to its smooth basis functions. Compared to methods such as finite difference methods~\cite{huang2014numerical,2018arXiv180203770M}, isogeometric analysis enables representing various geometrical shapes exactly such as a disk. Isogeometric analysis offers a new way to solve the fractional differential equations and can be generalized to other nonlocal models.

The article is organized as follows. In \Cref{sect:pre}, we describe the isogeometric collocation method. In \Cref{sect:ns}, we propose an algorithm for solving the fractional Poisson problem based on the isogeometric collocation method, where we construct the interpolation matrix and the discrete fractional Laplacian matrix. In \Cref{sect:ne} we perform the numerical benchmarks on generalized eigenvalues and eigenfunctions of the fractional Laplacian operator. We are able to compare the numerical solutions with analytical solutions and access convergence rates. Specifically, we compare the accuracy of the proposed method with the finite element method on a per-degree-of-freedom basis. We also consider the application of our method to the fractional porous media equation. The numerical results show the potential of isogeometric analysis for solving nonlocal models. We discuss the limitations and future research directions in \Cref{sect:conc}.

For reference, we list all notation used in this paper in \Cref{tab:notation}. The codes for this work will be available at \url{https://github.com/kailaix/IGA.jl}. 

\begin{table}[hbtp]
  \begin{tabular}{l|cc}
  \toprule
    Notation & Description \\
    \midrule
    $s$& the fractional index  \\
   $c_{s,2}$ & the normalization constant in the fractional Laplacian \\
   $(-\Delta)^s$ & the fractional Laplacian \\
    $B_{i,p}$ & the B-spline basis function \\
    $N_{i,j}$ & the NURBS basis function \\
    $\mathcal{C}^k$ & the function space with up to $k$-th order continuous derivatives\\
    $w_i$ & weights associated with the NURBS basis functions \\
    $\tilde X_i, \tilde \bX_{ij}$ & control points in 1D and 2D\\
    $F, \mathbf{F}$ & mapping from the parameter space to the physical space\\
    $\mathbf{V}$ & the interpolation function space defined by the NURBS basis functions\\
    $\hat \bu_{ij}, \hat \bx_{ij}$ & the collocation points in the parameter and physical space\\
    $\rho, a$& the window function and the window size\\
    $(-\Delta)_h~((-\Delta)^s_h)$ &  the discrete~(fractional) Laplacian\\
    $\mathbf{M}$ & the interpolation matrix\\
    $\mathbf{L}$ & the discrete fractional Laplacian matrix\\
    \bottomrule
  \end{tabular}
  \caption{Notation used in this paper}
  \label{tab:notation}
\end{table}

\section{Isogeometric Collocation Method}\label{sect:pre}

In isogeometric analysis, the basis functions for approximating the solutions are non-uniform rational basis splines~(NURBS). These basis functions emanate from computer-aided geometric design~(CAGD) in contrast to the Lagrange finite element interpolation polynomials used in finite element analysis~\cite{temizer2011contact}. In this section, we present an overview of the isogeometric collocation method. For details on isogeometric analysis, see the appendix.

Consider the generalized boundary-value problem 
\begin{equation}\label{equ:modelproblem}
      \mathcal{L}u = f \mbox{ in } \Omega, \qquad
    u = 0 \mbox{ in } \Omega^c
\end{equation}
where the solution $u:\mathbb{R}^2\rightarrow \mathbb{R}$ has a compact support in $\Omega$ and $\mathcal{L}$ is the fractional Laplacian, i.e., $(-\Delta)^s$. $f$ is the source term. 

We introduce a NURBS representation of $\Omega$. The collocation points are comprised of a finite set in the parameter space $\{\hat\bu_i\}_{i\in\mathcal{I}}$, where $\mathcal{I} = \mathcal{I}_D \cup \mathcal{I}_L$ is divided into two distinct sets~\cite{auricchio2010isogeometric}. The points in $\mathcal{I}_D$ lie on $\partial \Omega$ while those in $\mathcal{I}_L$ are inside $\Omega$. Then, the isogeometric collocation method finds $u_h\in \mathbf{V}$ such that
\begin{align}
    \mathcal{L}u_h(\mathbf{F}(\hat\bu_i)) = f(\mathbf{F}(\hat\bu_i)) &\qquad i\in \mathcal{I}_L\\
    u_h(\mathbf{F}(\hat\bu_i)) = 0 &\qquad  i\in \mathcal{I}_D
\end{align}

A commonly used set of collocation points is derived from the Greville abscissae~\cite{johnson2005higher}. The Greville abscissae $\bar u_i$ are related to the knot vector $\{u_1, u_2, \ldots, u_{l+1}\}$ by
\begin{equation}
    \bar u_i = \frac{u_{i+1}+u_{i+2}+\ldots+u_{i+p}}{p}
\end{equation}
Analogously, in two dimension, we can construct the Greville abscissae $\bar u_i$, $\bar v_j$ for both dimensions and consider the tensor product
\begin{equation*}
    \hat \bu_{ij} =(\bar u_i, \bar v_j)\qquad \hat \bx_{ij} = \mathbf{F}(\hat \bu_{ij})
\end{equation*}
Then the isogeometric analysis collocation method  reads: find $u_h\in\mathbf{V}$ such that
\begin{align*}
    \mathcal{L}u_h(\hat \bx_{ij}) &= f(\hat \bx_{ij}) &\qquad i=2,3,\ldots,l_u-p-1;  j=2,3,\ldots,l_v-q-1\\
    u_h(\hat \bx_{ij}) &= 0 &\qquad  (i,j) \in \{1, l_u-p\}\times \{1,2,\ldots, l_v-q\} \cup \{1,2,\ldots, l_u-p\}\times \{1, l_v-q\} 
\end{align*}

\section{Numerical Scheme}\label{sect:ns}

\subsection{Singularity Subtraction}

In this section, we will consider the discretization of the fractional Laplacian operator. Assume the knot vectors for $u$ and $v$ are 
    $$\{u_1, u_2, \ldots, u_{l_u+1}\}\quad \{v_1, v_2, \ldots, v_{l_v+1}\}$$
and the degrees are $p$ and $q$ respectively.  Let $\hat\bu_{ij}$ be the collocation points derived from the Greville abscissae, and $\hat\bx_{ij} = \mathbf{F}(\hat\bu_{ij})$. 

Assume that $u(\bx)\in\mathcal{C}^4$, we compute the principal value integral 
\[\mathrm{p.v.}~\int_{\mathbb{R}^2} \frac{u(\bx)-u(\by)}{|\bx-\by|^{2+2s}}  \]
using the singularity subtraction method~\cite{2018arXiv180203770M}
\begin{equation}\label{equ:111}
\int_{\mathbb{R}^2}  \frac{u(\bx)-u(\by)}{|\bx-\by|^{2+2s}} d\by =\int_{\mathbb{R}^2} \frac{u(\bx)-u(\by) + \rho(|\bx-\by|) g_\bx(\by)}{|\bx-\by|^{2+2s}} d\by 
- \int_{\mathbb{R}^2}  \frac{\rho(|\bx-\by|)g_\bx(\by)}{|\bx-\by|^{2+2s}} d\by
\end{equation}
where $\rho$ is a window function defined by
\begin{equation}\label{equ:window}
    \rho(r) =
    \begin{cases}
    1-35\left(\frac{r}{a}\right)^4 + 84\left(\frac{r}{a}\right)^5 - 70\left(\frac{r}{a}\right)^6 + 20 \left(\frac{r}{a}\right)^7 & r<a \\
        0 & \text{otherwise}
    \end{cases}
\end{equation}
Here $a>0$ is the window size and $g_\bx(\by)$ is the truncated Taylor expansion of $u(\by)-u(\bx)$
\begin{align}
    g_\bx(\by) &:= u_1(\bx) v_1 + u_2(\bx) v_2\notag \\
    &+ u_{11}(\bx)\frac{v_1^2}{2} + u_{22}(\bx)\frac{v_2^2}{2} + u_{12}(\bx)v_1v_2\notag \\
    &+ u_{111}(\bx) \frac{v_1^3}{6} + u_{112}(\bx)\frac{v_1^2v_2}{2} + u_{122}(\bx) \frac{v_1v_2^2}{2} + u_{222}(\bx) \frac{v_2^3}{6}\label{equ:gx}
\end{align}

Here we have used the abbreviation for the derivatives of $\mathbf{v}=(v_1,v_2) = \by-\bx$,
\begin{equation*}
    u_i := \frac{\partial u}{\partial \bx_i}, \;
    u_{ij} :=\frac{\partial^2 u}{\partial \bx_{ij}^2}, \;
    u_{ijk} :=\frac{\partial^3 u}{\partial \bx_{ijk}^3}, \; i,j,k \in \{1,2\}
\end{equation*}
The function $\rho$ was chosen such that $\rho(r)=1+\mathcal{O}(r^4)$, $r\rightarrow 0+$, $u(\bx)-u(\by) + \rho(|\bx-\by|) g_\bx(\by)\sim \mathcal{O}(|\bx-\by|^4)$, $\by\rightarrow \bx$, so that the integrand in  
\begin{equation}\label{equ:first}
  \int_{\mathbb{R}^2} \frac{u(\bx)-u(\by) + \rho(|\bx-\by|) g_\bx(\by)}{|\bx-\by|^{2+2s}} d\by 
\end{equation}
is continuous, and thus integrable. \Cref{fig:gx} shows $\frac{u(\bx)-u(\by) + \rho(|\bx-\by|) g_\bx(\by)}{|\bx-\by|^{2+2s}}$ and $\frac{|u(\bx)-u(\by)|}{|\bx-\by|^{2+2s}}$ for $s=0.3$. We can see that by singularity subtraction, the integrand is turned into a continuous function suitable for a standard numerical quadrature rule. The window function is shown in \Cref{fig:windowfunc}. 

\begin{figure}[htbp]
\centering
\includegraphics[width=0.49\textwidth]{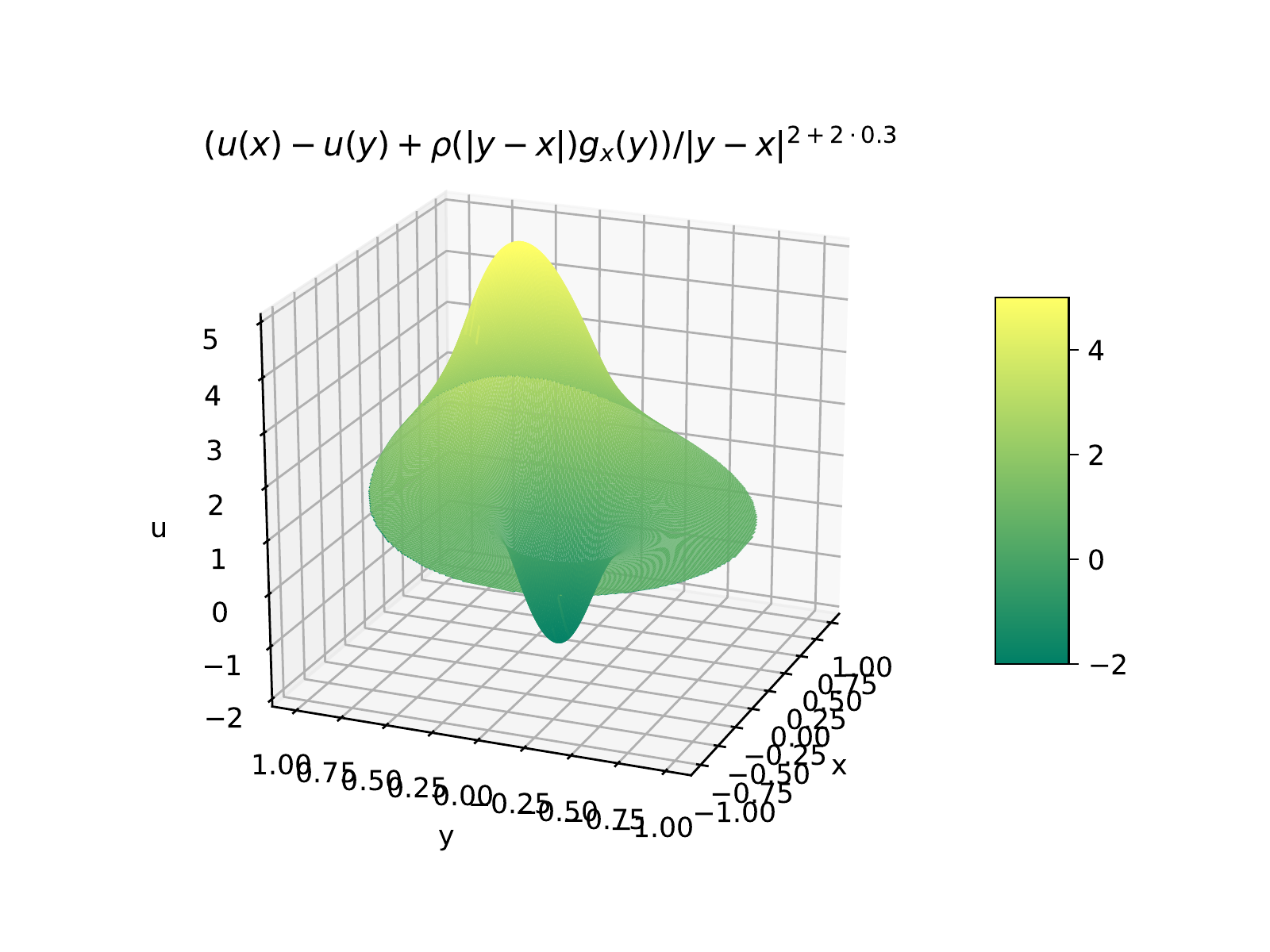}~
\includegraphics[width=0.49\textwidth]{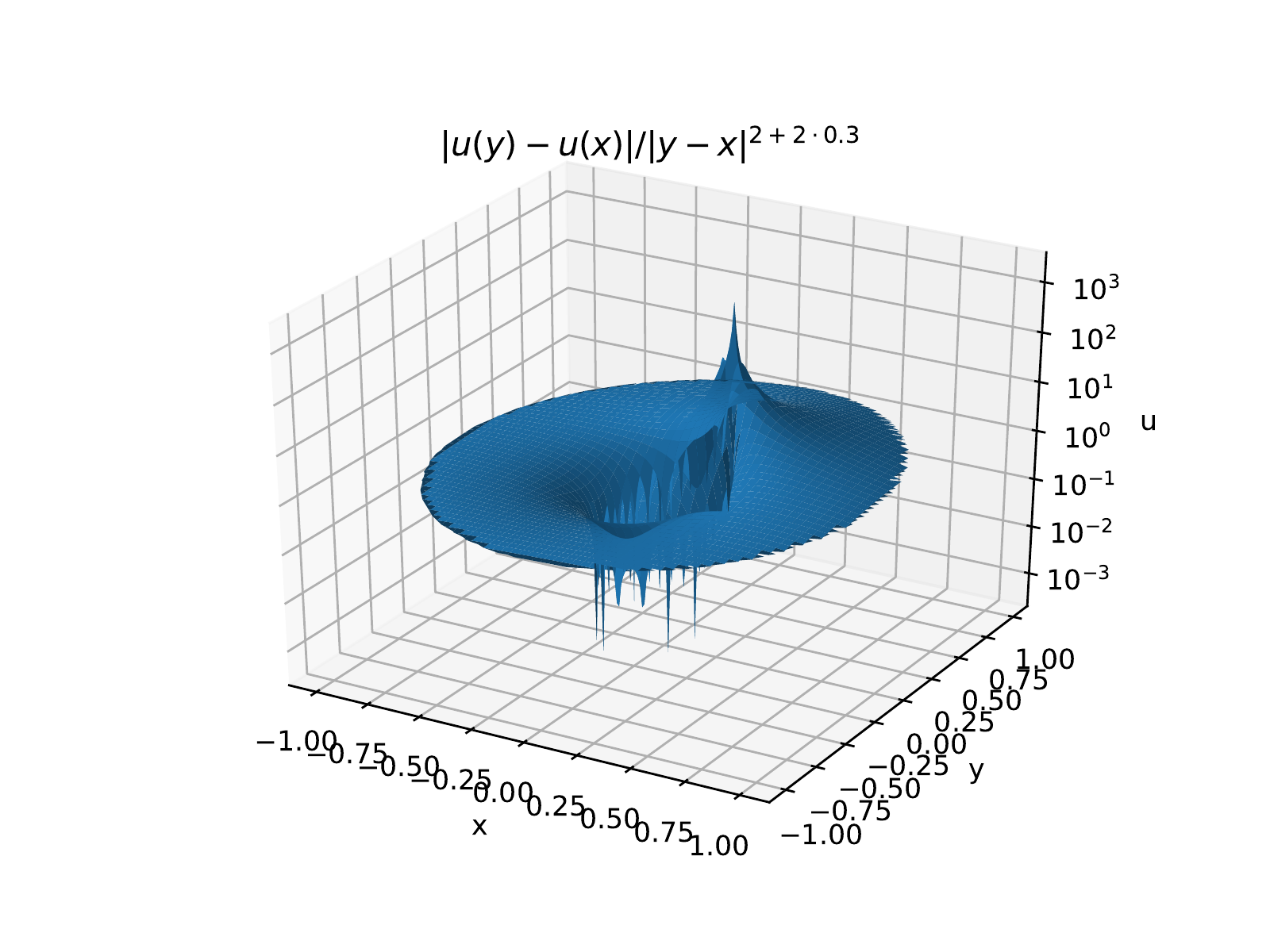}
\caption{$\frac{u(\bx)-u(\by) + \rho(|\bx-\by|) g_\bx(\by)}{|\bx-\by|^{2+2s}}$ and $\frac{|u(\bx)-u(\by)|}{|\bx-\by|^{2+2s}}$ for $s=0.3$ and $u(\bx)=(1-|\bx|^2)^{1+s}$. The function on the left is continuous while the one on the right has a singularity. By singularity subtraction, the integrand is turned into a continuous function suitable for a standard numerical quadrature rule. Note that the derivative of the integrand is not necessarily continuous.} 
\label{fig:gx}
\end{figure}

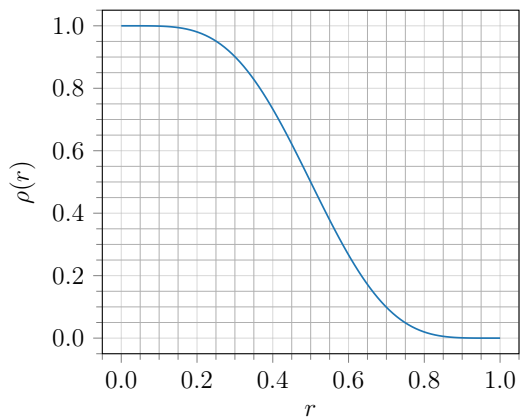
\begin{figure}[htpb]
\centering
\scalebox{0.8}{% This file was created by matplotlib2tikz v0.6.18.
\begin{tikzpicture}

\definecolor{color0}{rgb}{0.12156862745098,0.466666666666667,0.705882352941177}

\begin{axis}[
tick align=outside,
tick pos=left,
x grid style={white!69.01960784313725!black},
xlabel={$r$},
grid = both,
minor tick num=3,
    every major grid/.style={red, opacity=0.5},
xmin=-0.05, xmax=1.05,
xtick={-0.2,0,0.2,0.4,0.6,0.8,1,1.2},
xticklabels={â0.2,0.0,0.2,0.4,0.6,0.8,1.0,1.2},
y grid style={white!69.01960784313725!black},
ylabel={$\rho(r)$},
ymin=-0.05, ymax=1.05,
ytick={-0.2,0,0.2,0.4,0.6,0.8,1,1.2},
yticklabels={â0.2,0.0,0.2,0.4,0.6,0.8,1.0,1.2}
]
\addplot [thick, color0, forget plot]
table [row sep=\\]{%
0	1 \\
0.0101010101010101	0.999999987190997 \\
0.0202020202020202	0.999999603838709 \\
0.0303030303030303	0.999997093171168 \\
0.0404040404040404	0.999988167211075 \\
0.0505050505050505	0.999965126749867 \\
0.0606060606060606	0.999916221311124 \\
0.0707070707070707	0.999825224542672 \\
0.0808080808080808	0.999671201741958 \\
0.0909090909090909	0.999428447456598 \\
0.101010101010101	0.999066572311635 \\
0.111111111111111	0.998550719396774 \\
0.121212121212121	0.997841891700844 \\
0.131313131313131	0.996897373206895 \\
0.141414141414141	0.995671227359679 \\
0.151515151515152	0.99411485768782 \\
0.161616161616162	0.992177616405719 \\
0.171717171717172	0.989807447835146 \\
0.181818181818182	0.986951554473648 \\
0.191919191919192	0.983557074496166 \\
0.202020202020202	0.979571760407803 \\
0.212121212121212	0.974944649469387 \\
0.222222222222222	0.969626717393359 \\
0.232323232323232	0.963571507655626 \\
0.242424242424242	0.956735729589302 \\
0.252525252525253	0.949079819218734 \\
0.262626262626263	0.940568457556902 \\
0.272727272727273	0.931171041826135 \\
0.282828282828283	0.920862105771147 \\
0.292929292929293	0.909621685914672 \\
0.303030303030303	0.897435631259411 \\
0.313131313131313	0.884295854565639 \\
0.323232323232323	0.870200523931685 \\
0.333333333333333	0.855154193974496 \\
0.343434343434343	0.839167876449755 \\
0.353535353535354	0.822259050665401 \\
0.363636363636364	0.804451614529052 \\
0.373737373737374	0.785775777528596 \\
0.383838383838384	0.766267897376256 \\
0.393939393939394	0.745970262449592 \\
0.404040404040404	0.724930822538302 \\
0.414141414141414	0.703202870753277 \\
0.424242424242424	0.680844679774115 \\
0.434343434343434	0.657919095903273 \\
0.444444444444444	0.634493094659224 \\
0.454545454545455	0.610637301877279 \\
0.464646464646465	0.586425484495356 \\
0.474747474747475	0.561934015382658 \\
0.484848484848485	0.537241316722199 \\
0.494949494949495	0.512427286583205 \\
0.505050505050505	0.487572713416795 \\
0.515151515151515	0.462758683277802 \\
0.525252525252525	0.438065984617341 \\
0.535353535353535	0.413574515504645 \\
0.545454545454545	0.389362698122722 \\
0.555555555555556	0.365506905340775 \\
0.565656565656566	0.342080904096726 \\
0.575757575757576	0.319155320225887 \\
0.585858585858586	0.296797129246723 \\
0.595959595959596	0.275069177461694 \\
0.606060606060606	0.254029737550406 \\
0.616161616161616	0.233732102623742 \\
0.626262626262626	0.214224222471405 \\
0.636363636363636	0.195548385470948 \\
0.646464646464647	0.177740949334602 \\
0.656565656565657	0.160832123550252 \\
0.666666666666667	0.144845806025504 \\
0.676767676767677	0.129799476068314 \\
0.686868686868687	0.115704145434353 \\
0.696969696969697	0.102564368740595 \\
0.707070707070707	0.0903783140853314 \\
0.717171717171717	0.0791378942288645 \\
0.727272727272727	0.0688289581738672 \\
0.737373737373737	0.0594315424430967 \\
0.747474747474748	0.0509201807812545 \\
0.757575757575758	0.0432642704106954 \\
0.767676767676768	0.0364284923443918 \\
0.777777777777778	0.0303732826066252 \\
0.787878787878788	0.0250553505306144 \\
0.797979797979798	0.020428239592178 \\
0.808080808080808	0.0164429255038456 \\
0.818181818181818	0.0130484455263513 \\
0.828282828282828	0.010192552164888 \\
0.838383838383838	0.00782238359428788 \\
0.848484848484849	0.00588514231220216 \\
0.858585858585859	0.00432877264034914 \\
0.868686868686869	0.00310262679310824 \\
0.878787878787879	0.00215810829915597 \\
0.888888888888889	0.00144928060321092 \\
0.898989898989899	0.000933427688394772 \\
0.909090909090909	0.000571552543433995 \\
0.919191919191919	0.000328798258053098 \\
0.929292929292929	0.000174775457395526 \\
0.939393939393939	8.37786888041592e-05 \\
0.94949494949495	3.48732501009863e-05 \\
0.95959595959596	1.18327888856129e-05 \\
0.96969696969697	2.90682889669824e-06 \\
0.97979797979798	3.96161240701076e-07 \\
0.98989898989899	1.28090533735303e-08 \\
1	0 \\
};

\end{axis}

\end{tikzpicture}}
\caption{An example of the window function \Cref{equ:window} for $a=1$. The function behaves like $1-\mathcal{O}(r^4)$ near the origin.}
\label{fig:windowfunc}
\end{figure}

Due to symmetry, the second part of \Cref{equ:111}is
\begin{equation}\label{equ:delta}
    \int_{\mathbb{R}^2}  \frac{\rho(|\bx-\by|)g_\bx(\by)}{|\bx-\by|^{2+2s}} d\by = \frac{1}{2}\Delta u(\bx) 
    \int_{\mathbb{R}^2} \frac{\rho(|\by|)\by_1^2}{|\by|^{2+2s}}d\by
\end{equation}

\subsection{Numerical Discretization}

We now adopt a standard numerical quadrature rule for computing \Cref{equ:first}. The idea is that for a function in 2D we can express it in the polar coordinates
\begin{equation}
    \int_{\mathbb{R}^2} f(\bx) d\bx = \int_0^\infty dr \int_{\mathbb{S}^1} rf(r\bm{\sigma})d\bm{\sigma}
\end{equation}
where $\mathbb{S}^1$ is the unit circle. The numerical quadrature is constructed from the tensor product of the Gauss Legendre quadrature rule in the radial direction and the trapezoidal rule in the axial direction, i.e.,
\begin{equation}\label{equ:quad}
    \bm{\xi}_{ij} = r_i\bm{\sigma}_j, \qquad
    w^{Q}_{ij} = \frac{2\pi w^{GL}_j r_i}{m}, \;\; i=1,2,\ldots,n; ~j=1,2,\ldots,m
\end{equation}
where $(r_i,w_{ij}^{GL})$ are Gauss Legendre quadrature points and weights in $[0,R]$ for a sufficient large $R$. Therefore, \Cref{equ:first} can then be discretized as 
\begin{equation*}
    \sum\limits_{i = 1}^n {\sum\limits_{j = 1}^m {\frac{2\pi w_{ij}^{GL}}{m}{{u({\bx}) - u({\bx} + {\bm{\xi} _{ij}}) + {{{\rho _i}r_i^2\Delta u({\bx})} \over 2}} \over {r_i^{1 + 2s}}}} }
\end{equation*}
here $\rho_{i} = \rho(r_i)$. For the Laplacian operator $\Delta$, we consider the fourth order discretization~\cite{gibou2005fourth}
\begin{equation}\label{equ:du}
    \Delta u(\bx) \approx  \Delta_h u(\bx) =\sum\limits_{i = 1}^2 {{{ - {1 \over {12}}u(\bx - 2{\mathbf{e}_i}h) + {4 \over 3}u(\bx - {\mathbf{e}_i}h) - {5 \over 2}u(\bx) + {4 \over 3}u(\bx + {\mathbf{e}_i}h) - {1 \over {12}}u(\bx + 2{\mathbf{e}_i}h)} \over {{h^2}}}} 
\end{equation}
here $\mathbf{e}_1=\begin{bmatrix}
    1\\
    0
\end{bmatrix}$, $\mathbf{e}_2=\begin{bmatrix}
    0\\
    1
\end{bmatrix}$. Note we require $u(\bx)\in \mathcal{C}^4$ to obtain the fourth order estimate of the discretization \Cref{equ:du}.

Define 
\[
A = 2\pi \sum\limits_{i = 1}^n {{{ {w_{ij}^{GL}}{}} \over {r_i^{1 + 2s}}}}, \qquad
B = \sum\limits_{i = 1}^n { {{{\pi {w_{ij}^{GL}}{\rho _i}} \over {r_i^{ 2s-1}}}} }  - \pi \int_0^a {{{\rho (r)} \over {{r^{2s - 1}}}}dr} 
\]
then the discretization of the fractional Laplacian operator is given as
\begin{equation}\label{equ:discretize}
(-\Delta)^su(\bx) \approx    (-\Delta)^s_hu(\bx) = c_{s, 2} \left(Au(\bx) - \sum\limits_{i = 1}^n {\sum\limits_{j = 1}^m {{{2\pi {w_{ij}^{GL}}} \over {mr_i^{1 + 2s}}}} } u(\bx+\bm{\xi}_{ij}) + B\Delta_h u(\bx)\right)
\end{equation}

We pick a sufficient large $R$ so that the computational domain $\Omega$ is contained in the convex hull of the numerical quadrature points. Only the terms such that $\bm{x}+\bm{\xi}_{ij}\in\Omega$ do not vanish in  the second summation \Cref{equ:discretize}. \Cref{fig:gaussquad} depicts such an example where $\Omega$ is the unit disk and $\bx+\bm{\xi}_{ij}$ corresponds to the non-vanishing terms for two $\bx$'s are shown.

\begin{figure}[htbp]
\centering
\scalebox{0.8}{\input{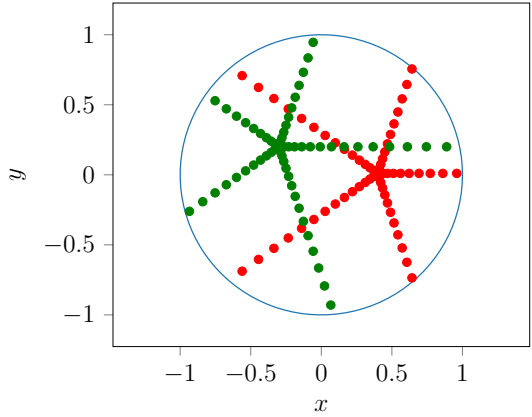}}
\caption{An example of quadrature points used in evaluating the second term in \Cref{equ:discretize}.}
\label{fig:gaussquad}
\end{figure}

\subsection{Interpolation Matrix and the Discrete Fractional Laplacian Matrix}
\label{sect:im}
Let us consider any point $\bx\in \Omega$, then there exists $\bu=\begin{bmatrix}
    u\\
    v
\end{bmatrix}$ in the parameter space such that $\mathbf{F}(\bu) = \bx$. The numerical solution $u_h\in\mathbf{V}$ can be written as
\begin{equation}\label{equ:uh}
    u_h(\bx) = u_h(\mathbf{F}^{-1}(\bu)) :=  \sum_{i=1}^{l_u-p}\sum_{j=1}^{l_v-q} N_{i,j}(u,v)c_{ij}
\end{equation}

 Since every NURBS basis function only spans $pq$ patches in the parameter space, no more than $pq$ terms in the summation in \Cref{equ:uh} are nonzero. If $\bx$ are the collocation points derived from Greville abscissae, we have
 \begin{equation}
     u_h(\hat\bx) = \mathbf{M}\mathbf{c}
 \end{equation}
 where $\hat\bx=\texttt{vec}(\{\hat \bx_{ij}\}_{i=1,2,\ldots, l_u-p;~j=1,2,\ldots, l_v-q})$, $\mathbf{c}=\texttt{vec}(\{c_{ij}\})$ and $\mathbf{M}$ is a sparse matrix with at most $pq$ non-zeros per row. Here \texttt{vec} is the vectorization of a matrix, 
 \begin{equation*}
     \texttt{vec}(\{\mathbf{x}_{ij}\}_{ij}) = \begin{bmatrix}
         \bx_{11}, \ldots, \bx_{l_u-p,1}, \bx_{12}, \ldots, \bx_{l_u-p,2}, \ldots, \bx_{1,l_v-q}, \ldots, \bx_{l_u-p, l_v-q}
     \end{bmatrix}^T
 \end{equation*}
 
 We call $\mathbf{M}$ the \textit{interpolation matrix}, which maps the coefficients from the parameter space to pointwise values in the physical space and vice versa. 
 
To compute \Cref{equ:discretize}, assume that $u(\bx)$ is given in terms of the coefficients $\mathbf{c}$ in the parameter space, then we can evaluate the values of $u$ at $\bx$, $\bx+\bm{\xi}_{ij}$, $\bx\pm h\mathbf{e}_i$, and $\bx\pm 2h\mathbf{e}_i$ using \Cref{equ:uh}. The values are computed and plugged into \Cref{equ:discretize} to obtain the values $(-\Delta)^s_h u(\bx)$, which is linear in the coefficients $\bc$. If the collocation points are $\{\hat \bx_{ij}\}$, we denote the associated linear operator that can be expressed by a matrix $\mathbf{L}$, and then we have
\begin{equation}
    (-\Delta)_h^s u(\hat\bx) = \mathbf{L}\mathbf{c}
\end{equation}

The interpolation matrix  $\mathbf{M}$ and the discrete fractional Laplacian matrix $\mathbf{L}$ can be precomputed. \Cref{alg:poisson} shows the algorithm for solving the Poisson equation using the matrices $\mathbf{M}$ and $\mathbf{L}$ mentioned above. 

\begin{algorithm}[htpb]
\begin{algorithmic}[1]
\STATE \textbf{Input}: $f$ \textbf{\;\;Output}: $u_h(\hat\bx_{ij})$ 
\STATE Precompute $\mathbf{L}$, $\mathbf{M}$
\STATE $\mathbf{f}\gets \texttt{vec}(\{\mathbf{f}_{ij}\}_{ij})$ where $\mathbf{f}_{ij} = f(\hat\bx_{ij})$
\STATE $\mathbf{f}[\mathcal{I}_D]\gets 0$ $\qquad$\COMMENT{Recall that $\mathcal{I}_D$ are the degrees of freedom associated with boundaries.}
\STATE $\mathbf{L}[\mathcal{I}_D,:]=\mathbf{M}[\mathcal{I}_D,:]$$\qquad$\COMMENT{Impose the zero boundary conditions.}
\STATE $\mathbf{c}\gets \mathbf{L}^{-1}\mathbf{f}$
\RETURN $\mathbf{M}\bc$
\end{algorithmic}
\caption{Solving the fractional Poisson problem using the interpolation matrix and the discrete fractional Laplacian matrix.}
\label{alg:poisson}
\end{algorithm}

\subsection{Convergence and Error Bounds}

The choice $m$ and $n$ in the numerical quadrature rule \Cref{equ:quad} as well the choice of step size $h$ in \Cref{equ:du}   affects the accuracy of the numerical scheme. As we show in the numerical examples, these parameters will influence the accuracy  while having little impact on the convergence rate. 

The approximation ability depends on the number of knots, the degrees, and the mesh structure of the NURBS surface. It is shown that under appropriate assumptions, the NURBS space on the physical domain delivers the optimal rate of the convergence, similar to the finite element space of degree $p$~\cite{bazilevs2006isogeometric}. The following lemma gives the global error estimate
\begin{lemma}\label{thm:1}
    Let $k$ and $l$ be integer indices with $0\leq k\leq l\leq p+1$, we have
    \begin{equation*}
        \sum_{K\in \mathcal{K}_h} \|v-\Pi_{\mathcal{V}_h}v\|^2_{\mathcal{H}_h^k(K)} \leq C_{\mathrm{shape}}
        \sum_{K\in \mathcal{K}_h} h_K^{2(l-k)} \sum_{i=0}^l \|\nabla F \|^{2(i-l)}_{L^\infty(F^{-1}(K))}|v|^2_{H^i(K)}
        \qquad \forall v \in H^l(\Omega)
    \end{equation*}
\end{lemma}
For the proof the readers are referred to~\cite{bazilevs2006isogeometric}. Here $\mathcal{K}_h$ are the patches in the parametric space, $\mathcal{V}_h$ is the NURBS space, $C_{\mathrm{shape}}$ is a constant depending on the structure of the NURBS surface, $H^i$ is the standard Sobolev space, $h_K$ is the diameter of $K$, and $\mathcal{H}^k_h(K)$ are patches in the physical space endowed with its own norms. The lemma indicates that if $v$ is sufficiently smooth, increasing the degree of the NURBS surface~($p$-refinement) or the number of control points~($h$-refinement) will yield better approximations. However, it may not be the case for less smooth functions, where are often encountered in the PDEs involving the fractional Laplacian~\cite{ros2014dirichlet}. 

We assume that the domain $\Omega$ satisfies the exterior ball condition, i.e., there exists a positive radium $\rho_0$ such that all the points on $\partial \Omega$ can be touched by some exterior ball of radius $\rho_0$. The solution to the fractional Poisson equation (\Cref{equ:modelproblem}) with $\mathcal{L}=(-\Delta)^s$ is only H\"older continuous according to Corollary 1.6 in~\cite{ros2014dirichlet}.
\begin{lemma}\label{thm:v-is-Calpha}
Let $\Omega$ be a bounded $C^{1,1}$ domain satisfying the exterior ball condition, $f\in L^\infty(\Omega)$, $u$ be a solution of \Cref{equ:modelproblem} with $\mathcal{L}=(-\Delta)^s$. Then, 
\[
    u \in C^{0,s}(\mathbb{R}^2) \quad \text{and} \quad
    \|u\|_{C^{0,s}} \leq C\|f\|_{L^\infty(\Omega)}
\]
where $C$ is a constant depending only on $\Omega$ and $s$.
\end{lemma}
The lemma indicates that the solution is H\"older continuous. In fact, it is also shown in \cite{ros2014dirichlet} that $\frac{u}{\delta^s}$ is continuous in $\bar\Omega$, i.e., the solution is continuous but not continuously differentiable near the boundary. For the case $s>\frac{1}{2}$, we have the following estimate
\begin{theorem}\label{theorem:main}
    Assume that $\Omega$ is a bounded $C^{1,1}$ domain satisfying the exterior ball condition, $f\in L^\infty(\Omega)$,  $u$ be a solution of \Cref{equ:modelproblem} with $\mathcal{L}=(-\Delta)^s$, $s>\frac{1}{2}$. In addition, assume that the mesh is quasi-uniform, i.e., there exists $\gamma_1\leq 1 \leq \gamma_2$, such that
    \begin{equation}
        \gamma_1 h \leq h_K \leq \gamma_2 h \quad \forall K \in \mathcal{K}_h
    \end{equation}
     then there exists $C>0$ such that
    \begin{equation}
        \sum_{K\in \mathcal{K}_h} \|u-\Pi_{\mathcal{V}_h}u\|^2_{L^2(\Omega)} \lesssim h^2
    \end{equation} 
\end{theorem}

To prove \Cref{theorem:main}, we need the following embedding lemma~\cite{slidespd48:online}
\begin{lemma}\label{thm:helper}
    Assume that $\Omega$ is a bounded $C^{1,1}$ domain, $m$, $k\in\mathbb{N}_0$, $p\in [1,\infty)$, $\alpha\in [0,1]$. 
    \begin{equation}
        m -\frac{n}{p} \leq k+\alpha \quad \alpha\neq 0, 1
    \end{equation}
    then $W^{m,p}(\Omega)\subset C^{k,\alpha}(\bar \Omega)$ and there is a constant $C>0$, s.t.
    \begin{equation}
        \|u\|_{W^{m,p}(\Omega)} \leq C\|u\|_{C^{k,\alpha}(\Omega)}
    \end{equation}
    If $m -\frac{n}{p} < k+\alpha$, the embedding is compact. 
\end{lemma}

\begin{proof}[Proof of \Cref{theorem:main}]
    Due to \Cref{thm:v-is-Calpha}, the solution is $s$-H\"older continuous. Let $n=p=2$, $m=1$, $k=0$, $\alpha=s$ in \Cref{thm:helper}, we have 
$$\|u\|_{H^1}\leq C\|u\|_{C^{0,s}(\Omega)}$$
    for a constant $C>0$. We invoke \Cref{thm:1} with $l=1$, $k=0$, and thus have 
$$\sum_{K\in \mathcal{K}_h} \|u-\Pi_{\mathcal{V}_h}u\|^2_{L^2(\Omega)} \lesssim h^2$$
\end{proof}
Note for the case $s\leq \frac{1}{2}$ the proof is not applicable. Nevertheless, we observe numerically that we  obtain the $\mathcal{O}(h^2)$ convergence rate. 
     
Another source of error is the numerical error of solving the linear system $\mathbf{L}\mathbf{c}=\mathbf{f}$ in \Cref{alg:poisson}. The condition number of $\mathbf{L}$ grows as we increase the number of control points. The convergence of the iterative solver such as GMRES may be slow if no proper preconditioner is used. In fact, it is shown that for finite element methods, if a family of shape regular and globally quasi-uniform triangulations with maximal element size $h$ is used, the stiffness matrix satisfies~\cite{ainsworth2017towards}
$$\kappa(\mathbf{L}) = Ch^{-2s}$$
In the paper, we have used Greville abscissae as the collocation points, which have been widely adopted as the default choice in the isogeometric analysis literature~\cite{reali2015introduction}.

\section{Numerical Experiments}\label{sect:ne}

\subsection{Numerical Benchmark}

For verification and benchmarking we consider the generalized eigenvalue problem for $(-\Delta)^s$ in a unit disk $\Omega\subset \mathbb{R}^d$, with a zero condition in the complement of $\Omega$
\begin{equation}\label{equ:ben}
    \begin{cases}
        (-\Delta)^s \left((1-|\bx|^2)^s\varphi_n(\bx)\right) = \lambda_n \varphi_n(\bx) & \bx\in \Omega\\
        \varphi_n(\bx) = 0 & \bx\not\in\Omega
    \end{cases}
\end{equation}
where $n=0,1,\ldots$ The eigenvalues $\lambda_n$ and the eigenfunctions $\varphi_n(\bx)$ are given as~\cite{dyda2017eigenvalues}
\[
    \lambda_n = \frac{{{2^{2s}}\Gamma {{(s + n + 1)}^2}} }{ {{{(n!)}^2}} }\qquad
    \varphi_n(\bx) = {(-1)^n}P_n^{\left( {s,0} \right)}(2|\bx|^2 - 1)
\]
We assume the right hand side $\lambda_n \varphi_n(\bx)$ is given and we apply the proposed algorithm \Cref{alg:poisson} to obtain the numerical solution $u_h(\hat\bx_{ij})$. The error is computed using 
\begin{equation}\label{equ:error}
    \mathrm{error} = \sqrt{\frac{\sum_{i,j} \left|u_h(\hat\bx_{ij})-(1-|\hat\bx_{ij}|^2)^s\varphi_n(\hat\bx_{ij})\right|^2}{(l_u-p)(l_v-q)}}
\end{equation}
here $(1-|\hat\bx_{k}|^2)^s\varphi_n(\hat\bx_{k})$ is the exact solution at $\hat\bx_k$.

The numerical experiments are carried out with parameters $s=0.8$, $a=0.1$, $h=0.001$, $R=20$, $p=q=2$, $n=1$, $2$, $3$, $4$, $5$. In the first case we use $m=20$, $n=1000$ while in the second case we use $m=40$, $n=5000$ ($m$ and $n$ are the number of quadrature points in the axial and radial directions). \Cref{fig:1} shows the convergence results for \Cref{equ:ben}. The first column shows the reference solution. The second column shows the convergence of the algorithm with respect the degrees of freedom $N$ by varying the number of refinements for different number of quadrature points. We have obtained monotonous convergence and observed $\mathcal{O}(N^{-1})$ convergence for all cases, where $N$ is the degrees of freedom. $m$ and $n$ influence the final error while having little impact on the convergence order. Larger numbers of quadrature points yield more accurate results. 

Note that the exact solutions $(1-|\bx|^2)^s\varphi_n(\bx)$ are continuous but not $\mathcal{C}^1$ on the boundary $|\bx|=1$. However, we observe no difficulty in applying the proposed algorithm even though we have required $\mathcal{C}^3$ for the singularity subtraction and $\mathcal{C}^4$ for the fourth-order discretization of the Laplacian operator. This is because the collocation points derived from the Greville abscissae are usually denser near the boundary and therefore mitigate accuracy loss partially.  

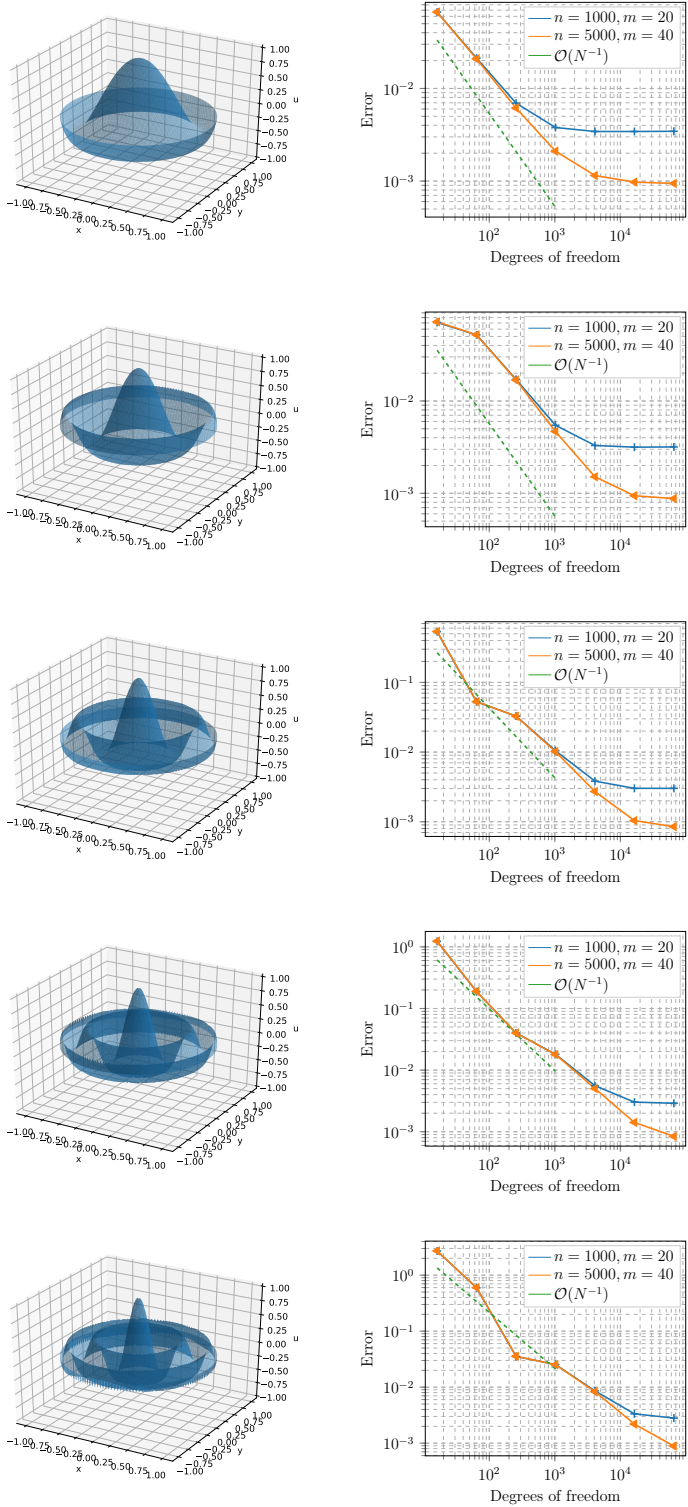
\begin{figure}[htpb] % \usepackage{float}
\centering
\includegraphics[width=0.33\textwidth,keepaspectratio]{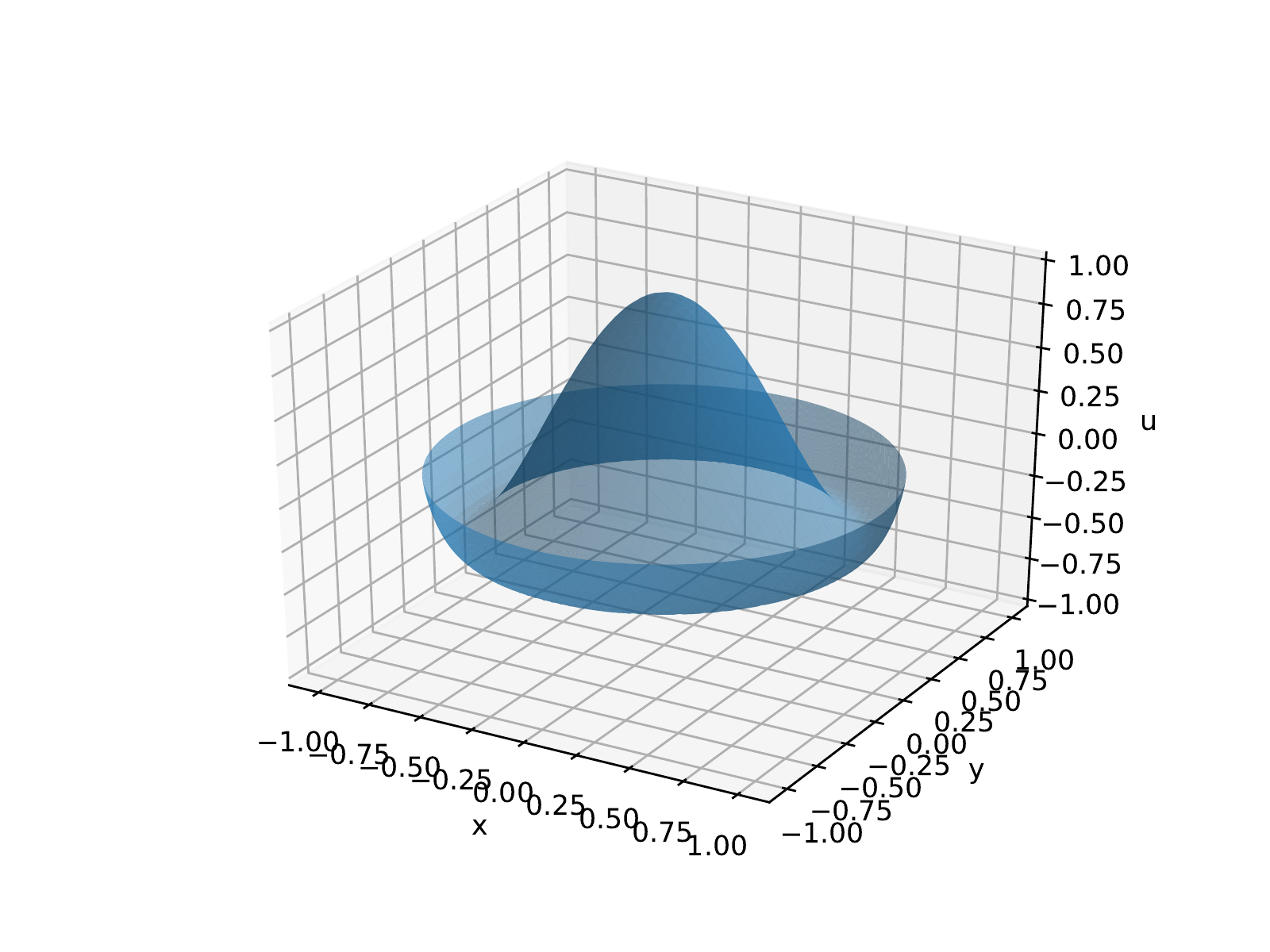}~
\scalebox{0.5}{% This file was created by matplotlib2tikz v0.6.18.
\begin{tikzpicture}

\definecolor{color0}{rgb}{0.12156862745098,0.466666666666667,0.705882352941177}
\definecolor{color1}{rgb}{1,0.498039215686275,0.0549019607843137}
\definecolor{color2}{rgb}{0.172549019607843,0.627450980392157,0.172549019607843}

\begin{axis}[
grid = both,
grid style={dashed,gray},
legend cell align={left},
legend entries={{$n=1000, m=20$},{$n=5000, m=40$},{$\mathcal{O}(N^{-1})$}},
legend style={draw=white!80.0!black},
tick align=outside,
tick pos=left,
x grid style={white!69.01960784313725!black},
xlabel={Degrees of freedom},
xmin=10.5560632861832, xmax=99334.0009028252,
xmode=log,
xtick={1,10,100,1000,10000,100000,1000000},
xticklabels={${10^{0}}$,${10^{1}}$,${10^{2}}$,${10^{3}}$,${10^{4}}$,${10^{5}}$,${10^{6}}$},
y grid style={white!69.01960784313725!black},
ylabel={Error},
ymin=0.000408727467761116, ymax=0.0849894054772691,
ymode=log,
ytick={1e-05,0.0001,0.001,0.01,0.1,1},
yticklabels={${10^{-5}}$,${10^{-4}}$,${10^{-3}}$,${10^{-2}}$,${10^{-1}}$,${10^{0}}$}
]
\addlegendimage{no markers, color0}
\addlegendimage{no markers, color1}
\addlegendimage{no markers, color2}
\addplot [very thick, color0, mark=+, mark size=3, mark options={solid}]
table [row sep=\\]{%
16	0.0666813360271643 \\
64	0.0212246498265138 \\
256	0.00695208588490713 \\
1024	0.00378648871048879 \\
4096	0.00343042478860083 \\
16384	0.00343012674368151 \\
65536	0.00344666191777939 \\
};
\addplot [very thick, color1, mark=triangle*, mark size=3, mark options={solid,rotate=90}]
table [row sep=\\]{%
16	0.0666647746020367 \\
64	0.0207973337117658 \\
256	0.00613583829457505 \\
1024	0.00210176876439697 \\
4096	0.00114790778341626 \\
16384	0.000974478627053965 \\
65536	0.000942320039038769 \\
};
\addplot [very thick, color2, dashed]
table [row sep=\\]{%
16	0.0333406680135821 \\
64	0.00833516700339553 \\
256	0.00208379175084888 \\
1024	0.000520947937712221 \\
};
\path [draw=black, fill opacity=0] (axis cs:0,0.000408727467761116)
--(axis cs:0,0.0849894054772691);

\path [draw=black, fill opacity=0] (axis cs:1,0.000408727467761116)
--(axis cs:1,0.0849894054772691);

\path [draw=black, fill opacity=0] (axis cs:10.5560632861832,0)
--(axis cs:99334.0009028252,0);

\path [draw=black, fill opacity=0] (axis cs:10.5560632861832,1)
--(axis cs:99334.0009028252,1);

\end{axis}

\end{tikzpicture}}
\includegraphics[width=0.33\textwidth,keepaspectratio]{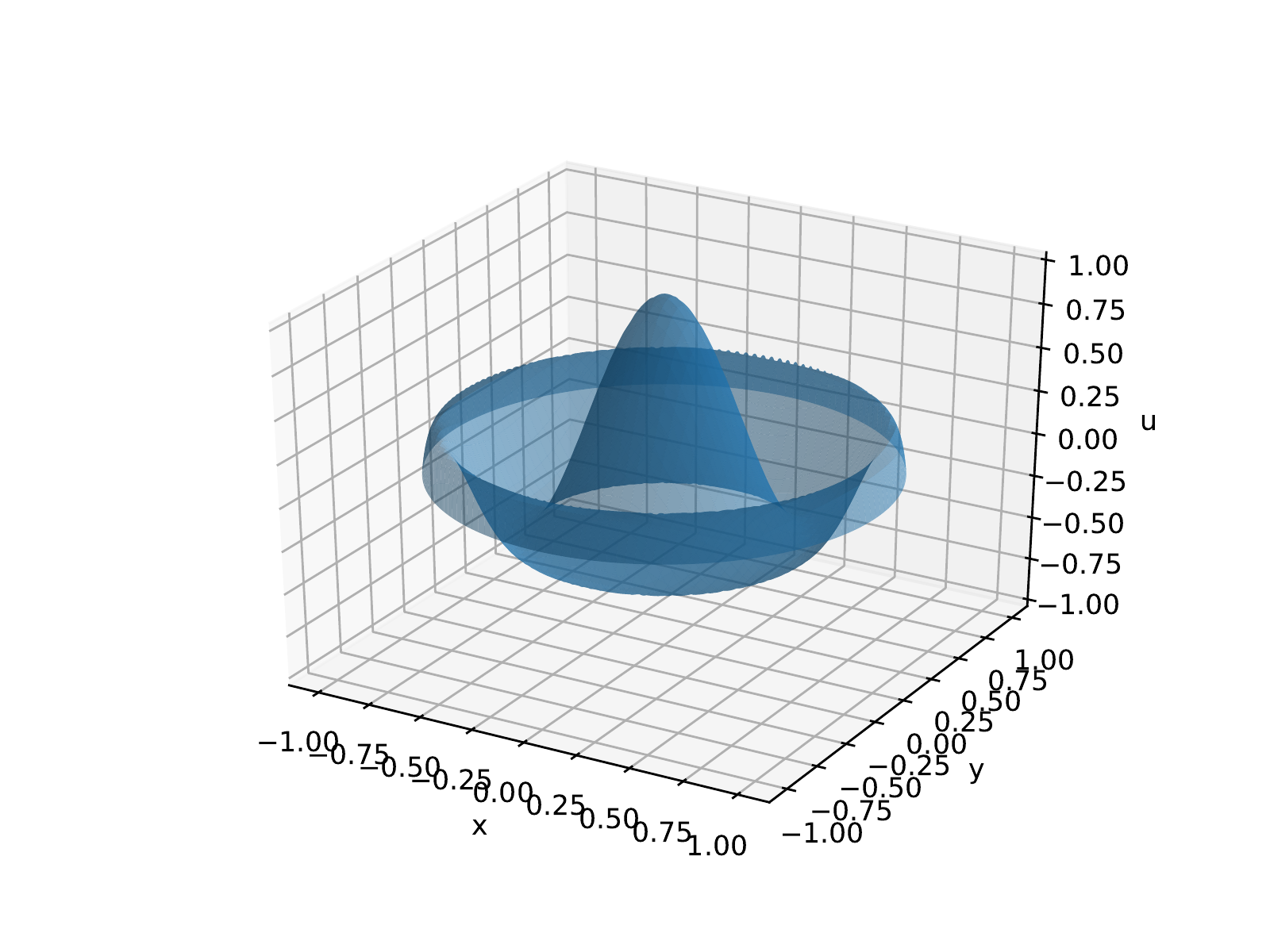}~
\scalebox{0.5}{% This file was created by matplotlib2tikz v0.6.18.
\begin{tikzpicture}

\definecolor{color0}{rgb}{0.12156862745098,0.466666666666667,0.705882352941177}
\definecolor{color1}{rgb}{1,0.498039215686275,0.0549019607843137}
\definecolor{color2}{rgb}{0.172549019607843,0.627450980392157,0.172549019607843}

\begin{axis}[
grid = both,
grid style={dashed,gray},
legend cell align={left},
legend entries={{$n=1000, m=20$},{$n=5000, m=40$},{$\mathcal{O}(N^{-1})$}},
legend style={draw=white!80.0!black},
tick align=outside,
tick pos=left,
x grid style={white!69.01960784313725!black},
xlabel={Degrees of freedom},
xmin=10.5560632861832, xmax=99334.0009028252,
xmode=log,
xtick={1,10,100,1000,10000,100000,1000000},
xticklabels={${10^{0}}$,${10^{1}}$,${10^{2}}$,${10^{3}}$,${10^{4}}$,${10^{5}}$,${10^{6}}$},
y grid style={white!69.01960784313725!black},
ylabel={Error},
ymin=0.000434346659427115, ymax=0.0918896785865606,
ymode=log,
ytick={1e-05,0.0001,0.001,0.01,0.1,1},
yticklabels={${10^{-5}}$,${10^{-4}}$,${10^{-3}}$,${10^{-2}}$,${10^{-1}}$,${10^{0}}$}
]
\addlegendimage{no markers, color0}
\addlegendimage{no markers, color1}
\addlegendimage{no markers, color2}
\addplot [very thick, color0, mark=+, mark size=3, mark options={solid}]
table [row sep=\\]{%
16	0.0709165876586422 \\
64	0.0520244838926537 \\
256	0.0171552130316527 \\
1024	0.00545920873206068 \\
4096	0.0032973583409459 \\
16384	0.00316065650121299 \\
65536	0.00317909881903189 \\
};
\addplot [very thick, color1, mark=triangle*, mark size=3, mark options={solid,rotate=90}]
table [row sep=\\]{%
16	0.0720386155016177 \\
64	0.0519762461879423 \\
256	0.0168597531837357 \\
1024	0.00465942984287403 \\
4096	0.00150949054331418 \\
16384	0.000941778392023012 \\
65536	0.000877022854153897 \\
};
\addplot [very thick, color2, dashed]
table [row sep=\\]{%
16	0.0354582938293211 \\
64	0.00886457345733027 \\
256	0.00221614336433257 \\
1024	0.000554035841083142 \\
};
\path [draw=black, fill opacity=0] (axis cs:0,0.000434346659427115)
--(axis cs:0,0.0918896785865606);

\path [draw=black, fill opacity=0] (axis cs:1,0.000434346659427115)
--(axis cs:1,0.0918896785865606);

\path [draw=black, fill opacity=0] (axis cs:10.5560632861832,0)
--(axis cs:99334.0009028252,0);

\path [draw=black, fill opacity=0] (axis cs:10.5560632861832,1)
--(axis cs:99334.0009028252,1);

\end{axis}

\end{tikzpicture}}
\includegraphics[width=0.33\textwidth,keepaspectratio]{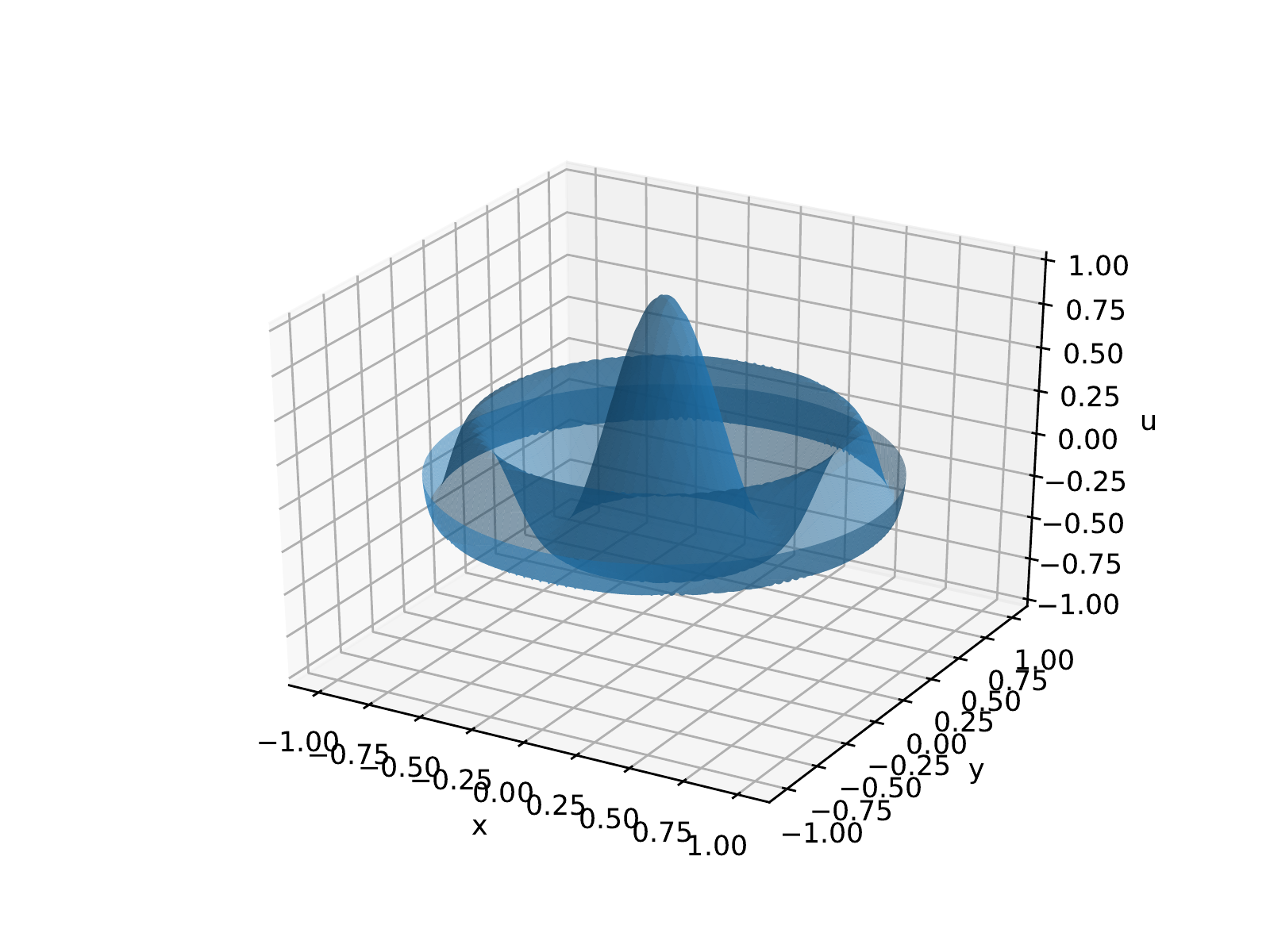}~
\scalebox{0.5}{% This file was created by matplotlib2tikz v0.6.18.
\begin{tikzpicture}

\definecolor{color0}{rgb}{0.12156862745098,0.466666666666667,0.705882352941177}
\definecolor{color1}{rgb}{1,0.498039215686275,0.0549019607843137}
\definecolor{color2}{rgb}{0.172549019607843,0.627450980392157,0.172549019607843}

\begin{axis}[
grid = both,
grid style={dashed,gray},
legend cell align={left},
legend entries={{$n=1000, m=20$},{$n=5000, m=40$},{$\mathcal{O}(N^{-1})$}},
legend style={draw=white!80.0!black},
tick align=outside,
tick pos=left,
x grid style={white!69.01960784313725!black},
xlabel={Degrees of freedom},
xmin=10.5560632861832, xmax=99334.0009028252,
xmode=log,
xtick={1,10,100,1000,10000,100000,1000000},
xticklabels={${10^{0}}$,${10^{1}}$,${10^{2}}$,${10^{3}}$,${10^{4}}$,${10^{5}}$,${10^{6}}$},
y grid style={white!69.01960784313725!black},
ylabel={Error},
ymin=0.000615054344343345, ymax=0.743116534847316,
ymode=log,
ytick={1e-05,0.0001,0.001,0.01,0.1,1,10},
yticklabels={${10^{-5}}$,${10^{-4}}$,${10^{-3}}$,${10^{-2}}$,${10^{-1}}$,${10^{0}}$,${10^{1}}$}
]
\addlegendimage{no markers, color0}
\addlegendimage{no markers, color1}
\addlegendimage{no markers, color2}
\addplot [very thick, color0, mark=+, mark size=3, mark options={solid}]
table [row sep=\\]{%
16	0.534444231513862 \\
64	0.0526460192025517 \\
256	0.0328789321383582 \\
1024	0.01045335374436 \\
4096	0.00382555782498392 \\
16384	0.00301652988601027 \\
65536	0.00301428300182034 \\
};
\addplot [very thick, color1, mark=triangle*, mark size=3, mark options={solid,rotate=90}]
table [row sep=\\]{%
16	0.538219583623174 \\
64	0.0528921541220693 \\
256	0.0327439548728642 \\
1024	0.0101215384088496 \\
4096	0.00271239407559404 \\
16384	0.00103922395217698 \\
65536	0.000849201825831769 \\
};
\addplot [very thick, color2, dashed]
table [row sep=\\]{%
16	0.267222115756931 \\
64	0.0668055289392328 \\
256	0.0167013822348082 \\
1024	0.00417534555870205 \\
};
\path [draw=black, fill opacity=0] (axis cs:0,0.000615054344343345)
--(axis cs:0,0.743116534847316);

\path [draw=black, fill opacity=0] (axis cs:1,0.000615054344343345)
--(axis cs:1,0.743116534847316);

\path [draw=black, fill opacity=0] (axis cs:10.5560632861832,0)
--(axis cs:99334.0009028252,0);

\path [draw=black, fill opacity=0] (axis cs:10.5560632861832,1)
--(axis cs:99334.0009028252,1);

\end{axis}

\end{tikzpicture}}
\includegraphics[width=0.33\textwidth,keepaspectratio]{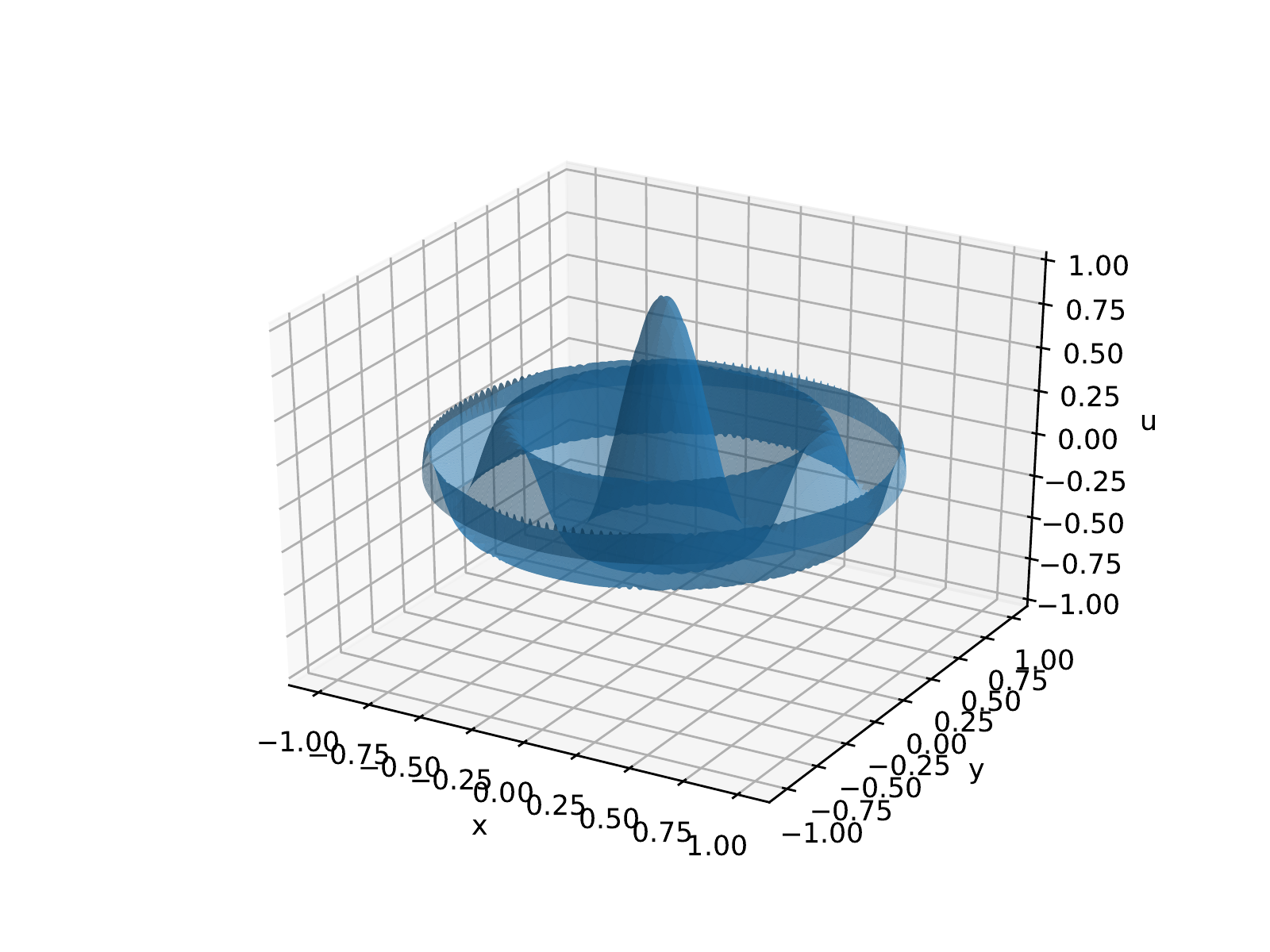}~
\scalebox{0.5}{% This file was created by matplotlib2tikz v0.6.18.
\begin{tikzpicture}

\definecolor{color0}{rgb}{0.12156862745098,0.466666666666667,0.705882352941177}
\definecolor{color1}{rgb}{1,0.498039215686275,0.0549019607843137}
\definecolor{color2}{rgb}{0.172549019607843,0.627450980392157,0.172549019607843}

\begin{axis}[
grid = both,
grid style={dashed,gray},
legend cell align={left},
legend entries={{$n=1000, m=20$},{$n=5000, m=40$},{$\mathcal{O}(N^{-1})$}},
legend style={draw=white!80.0!black},
tick align=outside,
tick pos=left,
x grid style={white!69.01960784313725!black},
xlabel={Degrees of freedom},
xmin=10.5560632861832, xmax=99334.0009028252,
xmode=log,
xtick={1,10,100,1000,10000,100000,1000000},
xticklabels={${10^{0}}$,${10^{1}}$,${10^{2}}$,${10^{3}}$,${10^{4}}$,${10^{5}}$,${10^{6}}$},
y grid style={white!69.01960784313725!black},
ylabel={Error},
ymin=0.000584053800768533, ymax=1.77896872831977,
ymode=log,
ytick={1e-05,0.0001,0.001,0.01,0.1,1,10,100},
yticklabels={${10^{-5}}$,${10^{-4}}$,${10^{-3}}$,${10^{-2}}$,${10^{-1}}$,${10^{0}}$,${10^{1}}$,${10^{2}}$}
]
\addlegendimage{no markers, color0}
\addlegendimage{no markers, color1}
\addlegendimage{no markers, color2}
\addplot [very thick, color0, mark=+, mark size=3, mark options={solid}]
table [row sep=\\]{%
16	1.22834873352337 \\
64	0.187814401576398 \\
256	0.0400050256270758 \\
1024	0.0180645185470789 \\
4096	0.00558367681816824 \\
16384	0.00302548497931177 \\
65536	0.00289372134414085 \\
};
\addplot [very thick, color1, mark=triangle*, mark size=3, mark options={solid,rotate=90}]
table [row sep=\\]{%
16	1.23542848681326 \\
64	0.18965773077752 \\
256	0.0399006695949209 \\
1024	0.017925669831132 \\
4096	0.00501558782503527 \\
16384	0.0014160255763014 \\
65536	0.000841014642542055 \\
};
\addplot [very thick, color2, dashed]
table [row sep=\\]{%
16	0.614174366761687 \\
64	0.153543591690422 \\
256	0.0383858979226054 \\
1024	0.00959647448065136 \\
};

\end{axis}

\end{tikzpicture}}
\includegraphics[width=0.33\textwidth,keepaspectratio]{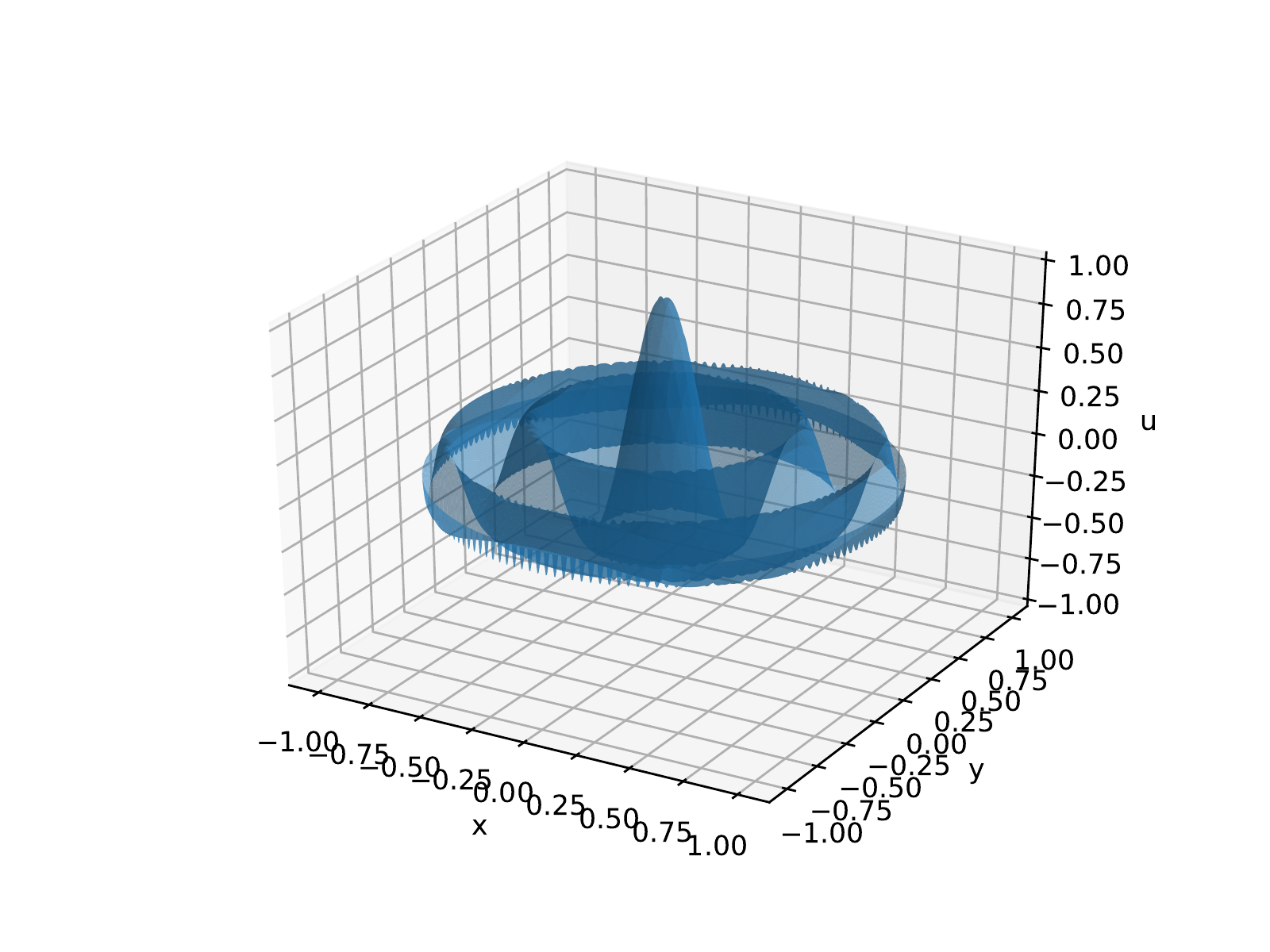}~
\scalebox{0.5}{% This file was created by matplotlib2tikz v0.6.18.
\begin{tikzpicture}

\definecolor{color0}{rgb}{0.12156862745098,0.466666666666667,0.705882352941177}
\definecolor{color1}{rgb}{1,0.498039215686275,0.0549019607843137}
\definecolor{color2}{rgb}{0.172549019607843,0.627450980392157,0.172549019607843}

\begin{axis}[
grid = both,
grid style={dashed,gray},
legend cell align={left},
legend entries={{$n=1000, m=20$},{$n=5000, m=40$},{$\mathcal{O}(N^{-1})$}},
legend style={draw=white!80.0!black},
tick align=outside,
tick pos=left,
x grid style={white!69.01960784313725!black},
xlabel={Degrees of freedom},
xmin=10.5560632861832, xmax=99334.0009028252,
xmode=log,
xtick={1,10,100,1000,10000,100000,1000000},
xticklabels={${10^{0}}$,${10^{1}}$,${10^{2}}$,${10^{3}}$,${10^{4}}$,${10^{5}}$,${10^{6}}$},
y grid style={white!69.01960784313725!black},
ylabel={Error},
ymin=0.000594479436530356, ymax=4.05960060057406,
ymode=log,
ytick={1e-05,0.0001,0.001,0.01,0.1,1,10,100},
yticklabels={${10^{-5}}$,${10^{-4}}$,${10^{-3}}$,${10^{-2}}$,${10^{-1}}$,${10^{0}}$,${10^{1}}$,${10^{2}}$}
]
\addlegendimage{no markers, color0}
\addlegendimage{no markers, color1}
\addlegendimage{no markers, color2}
\addplot [very thick, color0, mark=+, mark size=3, mark options={solid}]
table [row sep=\\]{%
16	2.70289301284659 \\
64	0.595038538959283 \\
256	0.0354336618186919 \\
1024	0.0253105209358127 \\
4096	0.00858900918445404 \\
16384	0.00333853187037988 \\
65536	0.00280876032421907 \\
};
\addplot [very thick, color1, mark=triangle*, mark size=3, mark options={solid,rotate=90}]
table [row sep=\\]{%
16	2.71765818446376 \\
64	0.599650547754297 \\
256	0.0356401967890222 \\
1024	0.0252580156424062 \\
4096	0.00831548909616023 \\
16384	0.00219092663093537 \\
65536	0.000888025245913609 \\
};
\addplot [very thick, color2, dashed]
table [row sep=\\]{%
16	1.3514465064233 \\
64	0.337861626605824 \\
256	0.084465406651456 \\
1024	0.021116351662864 \\
};

\end{axis}

\end{tikzpicture}}
\caption{Convergence results for \Cref{equ:ben}. The first column shows the reference solution. The second column shows the convergence of the algorithm with respect the degrees of freedom $N$ by varying the number of refinements for different number of quadrature points. The parameters $m$ and $n$ are the number of quadrature points in the axial and radial directions.}
\label{fig:1}
\end{figure}

\subsection{Comparison with FEM}

In this section, we compare the accuracy of the proposed algorithm with finite element analysis on a per-degree-of-freedom basis. We solve the same problem in \Cref{equ:ben} with $n=2$ and $s=0.5$ using isogeometric analysis as well as finite element analysis~\cite{acosta2017short}. The corresponding finite element analysis codes are made available by the authors.\footnote{We used the codes from \\ {\scriptsize\url{https://github.com/fbersetche/A-short-FE-implementation-for-a-2d-homogeneous-Dirichlet-problem-of-a-Fractional-Laplacian}}} For the finite element method, the error is measured as the mean squared error on the vertices $\tilde \bx_k$ of the mesh triangles, i.e.
$$\mathrm{error} = \sqrt{\frac{\sum_{k} \left|u_h(\bx_{k})-(1-|\bx_{k}|^2)^s\varphi_n(\bx_{k})\right|^2}{(l_u-p)(l_v-q)}}$$
In \Cref{fig:compare}, we report the log-scale plots of the errors for FEM and isogeometric analysis~(IGA) with $n=1000$, $m=20$ and $n=5000$, $m=40$. We can see that isogeometric analysis exhibits superior accuracy compared with finite element analysis for the fractional Laplacian problem. Besides, isogeometric analysis yields a better convergence rate without particular preprocessing such as the graded meshes adopted in \cite{acosta2017short}. We also point out the isogeometric analysis code is much easier to implement, while for finite element analysis we need to carefully treat the quadrature rules for the singular kernel. For example, \cite{acosta2017short} applied the Duffy-type transforms and implemented different quadrature rules according to the relative position of two triangle elements~(i.e., identical, sharing a vertex, sharing an edge and disjoint).

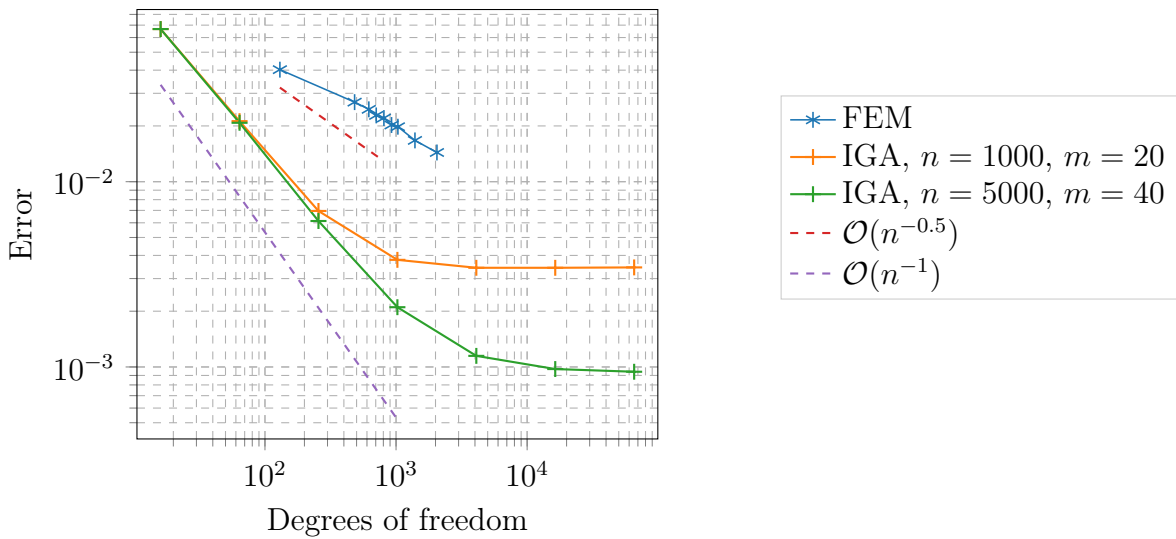
\begin{figure}[htbp]
\centering
\scalebox{1.0}{% This file was created by matplotlib2tikz v0.7.2.
\begin{tikzpicture}

\definecolor{color0}{rgb}{0.12156862745098,0.466666666666667,0.705882352941177}
\definecolor{color1}{rgb}{1,0.498039215686275,0.0549019607843137}
\definecolor{color2}{rgb}{0.172549019607843,0.627450980392157,0.172549019607843}
\definecolor{color3}{rgb}{0.83921568627451,0.152941176470588,0.156862745098039}
\definecolor{color4}{rgb}{0.580392156862745,0.403921568627451,0.741176470588235}

\begin{axis}[
grid = both,
grid style={dashed,gray},
legend cell align={left},
legend style={at = {(2.0,0.8)}, draw=white!80.0!black},
log basis x={10},
log basis y={10},
tick align=outside,
tick pos=left,
x grid style={white!69.01960784313725!black},
xlabel={Degrees of freedom},
xmin=10.5560632861832, xmax=99334.0009028252,
xmode=log,
y grid style={white!69.01960784313725!black},
ylabel={Error},
ymin=0.000408727467761116, ymax=0.0849894054772691,
ymode=log
]
\addplot [semithick, color0, mark=asterisk, mark size=3, mark options={solid}]
table {%
130 0.040254
483 0.026884
620 0.024585
706 0.022796
808 0.022064
924 0.02023
1033 0.01981
1394 0.016708
2051 0.01438
};
\addlegendentry{FEM}
\addplot [thick, color1, mark=+, mark size=3, mark options={solid}]
table {%
16 0.0666813360271643
64 0.0212246498265138
256 0.00695208588490713
1024 0.00378648871048879
4096 0.00343042478860083
16384 0.00343012674368151
65536 0.00344666191777939
};
\addlegendentry{IGA, $n=1000$, $m=20$}
\addplot [thick, color2, mark=+, mark size=3, mark options={solid}]
table {%
16 0.0666647746020367
64 0.0207973337117658
256 0.00613583829457505
1024 0.00210176876439697
4096 0.00114790778341626
16384 0.000974478627053965
65536 0.000942320039038769
};
\addlegendentry{IGA, $n=5000$, $m=40$}
\addplot [thick, color3, dashed]
table {%
130 0.0322032
483 0.0167069487366617
620 0.0147460214279496
706 0.0138187371119785
808 0.0129171000969587
};
\addlegendentry{$\mathcal{O}(n^{-0.5})$}
\addplot [thick, color4, dashed]
table {%
16 0.0333406680135822
64 0.00833516700339554
256 0.00208379175084888
1024 0.000520947937712221
};
\addlegendentry{$\mathcal{O}(n^{-1})$}
\end{axis}

\end{tikzpicture}}
\caption{Isogeometric analysis exhibits superior accuracy compared with finite element analysis for the fractional Laplacian problem.}
\label{fig:compare}
\end{figure}

\subsection{Application: Porous Medium Equation with the Fractional Laplacian}
 
 This application is concerned with the fractional porous media equation, which describes anomalous diffusion in the porous media~\cite{de2010fractional}
 \begin{equation}\label{equ:porous}
     \begin{cases}
         u_t + (-\Delta)^s(|u|^{m-1}u) = 0 & \bx\in \mathbb{R}^2,\; t>0\\
         u(\bx,0) = f(\bx) & \bx\in \mathbb{R}^2
     \end{cases}
 \end{equation}
 
The motivation of studying nonlocal models for diffusion in porous media is two folds: on the one hand, the traditional diffusion equation which leads to a Gaussian solution does not always agree with experimental results, and in some cases fractional models are required to match experiments; on the other hand, the fractional diffusion equation \mbox{\Cref{equ:porous}} can be derived from ``first principle'' using statistical mechanics, i.e., the fractional Laplacian naturally arises from macroscopic governing equations of particles undergoing $\alpha$-stable L\'evy motion. The associated parameter $s$ can be found from experimental data \mbox{\cite{gulian2019machine,maryshev2013adjoint,xu2018calibrating}}. In sum, the use of fractional Laplacian operator leads to more accurate modeling compared to integer-order differential equation approximations.
 
 We truncate the domain to $[-1,1]^2$ and only consider the nonnegative solutions $u\geq 0$. Define $t_n=(n-1)\Delta t$ to be the integration time $0\leq t_n\leq T$ and $\Delta t = \frac{T}{n_T}$. Let $u^n_{ij}$ be the numerical solution to $u(\hat\bx_{ij}, t_n)$ and denote $\mathbf{u}^n=\texttt{vec}(\{u^n_{ij}\}_{ij})$. We denote $(u^m)^n_{ij}=u(\hat\bx_{ij}, t_n)^m$. Consider the Crank-Nicolson discretization for \Cref{equ:porous}
$${{u_{ij}^{n + 1} - u_{ij}^n} \over {\Delta t}} + {1 \over 2}{( - \Delta )^s}({u^m})_{ij}^{n + 1} + {1 \over 2}{( - \Delta )^s}({u^m})_{ij}^n = 0$$
To linearize the equation, we consider the same technique used in~\cite{plociniczak2014approximation}
$$({u^m})_{ij}^{n + 1} \approx ({u^m})_{ij}^n + m({u^{m - 1}})_{ij}^n(u_{ij}^n - u_{ij}^{n - 1})$$
Using the matrices in \Cref{sect:im}, we have the discretized equation~(assuming $\mathbf{u}^{-1}=\mathbf{0}$)
\begin{align*}
    {{\mathbf{u}}^{n + 1}} &= {{\mathbf{u}}^n} - {{m\Delta t} \over 2}{\mathbf{F}^*}{{\mathbf{M}^*}^{ - 1}}\left( {({{\mathbf{u}}^{m - 1}})_{}^n \otimes ({\mathbf{u}}_{}^n - {\mathbf{u}}_{}^{n - 1})} \right) - {{\Delta t} \over 2}{\mathbf{F}^*}{{\mathbf{M}^*}^{ - 1}}({{\mathbf{u}}^m})_{}^n & n=1,2,\ldots, n_T\\
    \mathbf{u}^1 &= \mathbf{f}
\end{align*}
where $\bf=\texttt{vec}(\{f(\hat\bx_{ij})\}_{i,j\in \mathcal{I}_L})$, and $\otimes$ denotes the element-wise multiplication, and
$$\mathbf{M}^* = \mathbf{M}[\mathcal{I}_L, \mathcal{I}_L]\qquad \mathbf{L}^* = \mathbf{L}[\mathcal{I}_L, \mathcal{I}_L]$$

In the numerical example, we let $f(\bx) = \exp(-100|\bx|^2)$, $a=0.1$, $R=20$, $h=1000$, $n_T=1000$, $\Delta t=0.0001$, 1000 quadrature points in the radial direction and 20 in the axial direction. \Cref{fig:topval} shows the changes in the value $u_h(\mathbf{0}, t)$ for $0\leq t\leq 0.1$ and different $m$, $s$. Typically we do not have analytical solutions for general $m$ and $s$; however, in the special case when $m=1$ and $s=0.5$, we obtain the equation
\begin{equation*}
    \frac{\partial u}{\partial t} + (-\Delta)^{1/2} u = 0
\end{equation*}
the linear fractional heat equation has analytical solution through convolution with the explicit Poisson kernel in $\mathbb{R}^3_+$~($\mathbb{R}^2$ spatial and $\mathbb{R}_+$ temporal)~\cite{de2010fractional}
\begin{equation*}
    u(\bx, t) = c_{1/2,2} \int_{\mathbb{R}^2} \frac{tf(\by)}{(|\bx-\by|^2+t^2)^{3/2}}d\by
\end{equation*}
Particularly, we have in our case
\begin{equation}\label{equ:uexact}
    u(\mathbf{0},t) = 1-10\sqrt{\pi t^2}e^{100t^2}\mathrm{erfc}(10t)
\end{equation}
The equation \Cref{equ:uexact} can be used for verification and the values are plotted alongside the numerical results in \Cref{fig:topval}. We can see that the numerical result and the exact solution coincide for $t\in[0,0.1]$ and $s=0.2$, $m=1$. We can also see that the larger the $s$ or the smaller the $m$, the faster the diffusion is. Different values of $s$ and $m$ provide different profiles of the diffusion process and therefore form a powerful tool for modeling anomalous diffusion. \Cref{fig:final} shows the diffusion profile at $t=0.1$ for different cases. 

\begin{figure}[htbp]
\centering
\scalebox{1.0}{\input{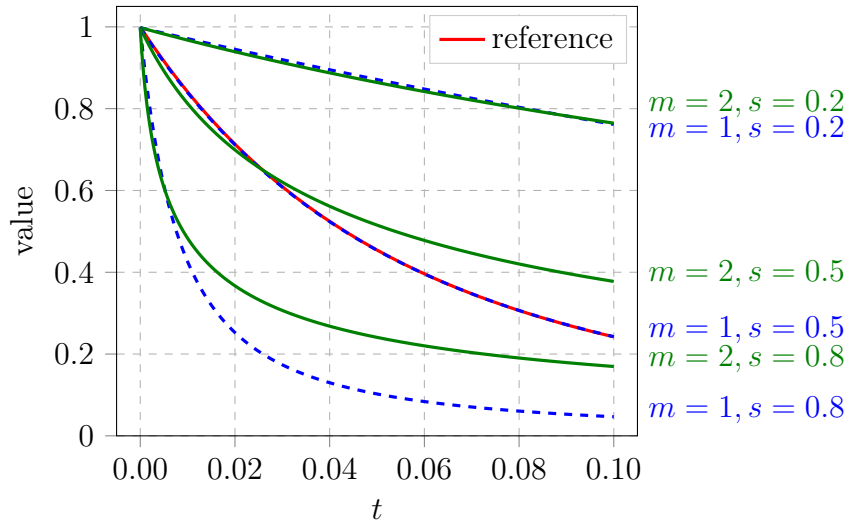}}
\caption{Evolution of $u_h(\mathbf{0}, t)$ in different settings. The numerical results coincide with the exact solution for $m=1$, $s=0.5$. The reference solution is given only for $m=1$, $s=0.5$, because analytical solutions are not available for the other parameters.}
\label{fig:topval}
\end{figure}

\begin{figure}[htpb] % \usepackage{float}
\centering
\includegraphics[width=0.33\textwidth,keepaspectratio]{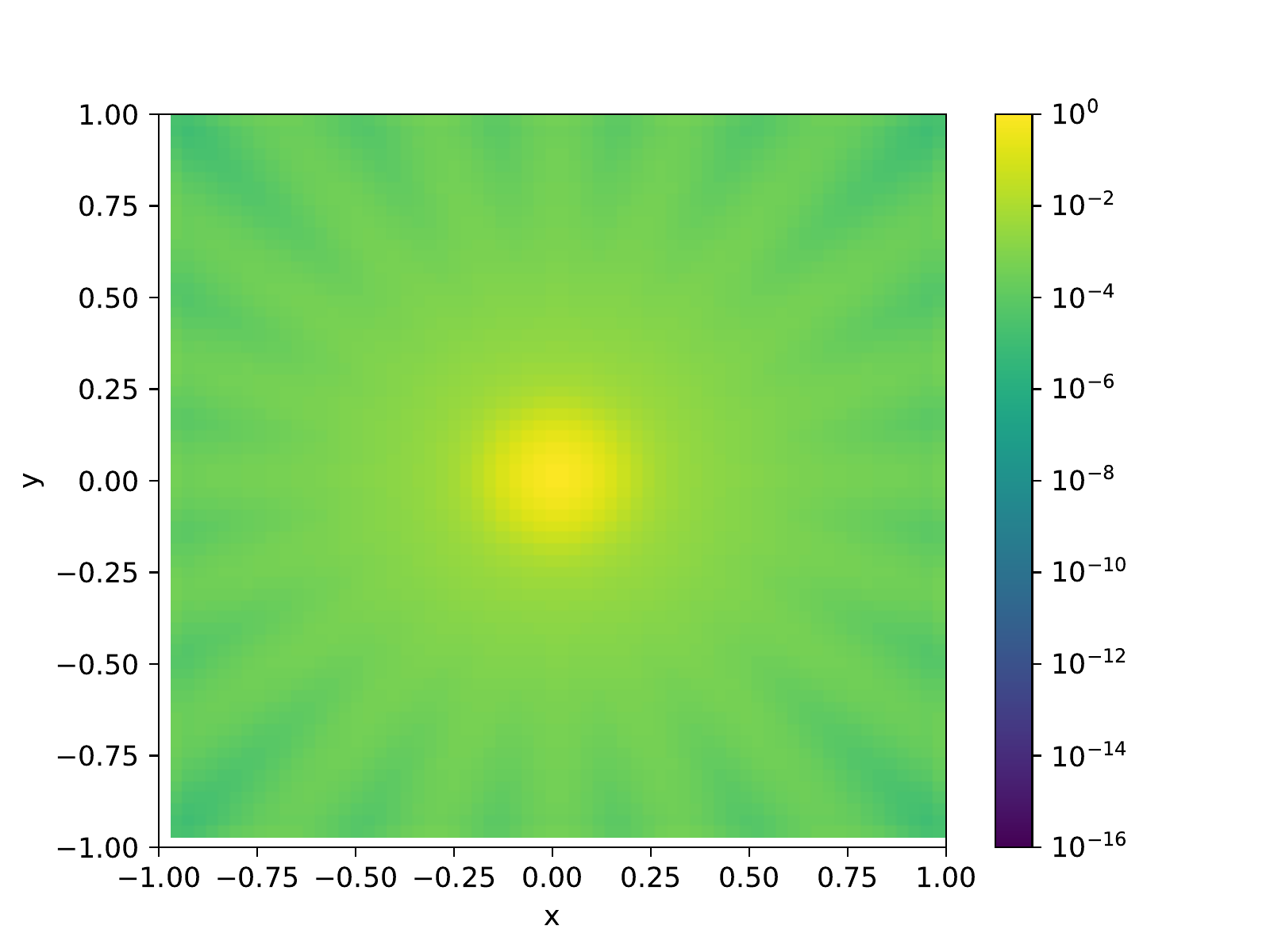}~
\includegraphics[width=0.33\textwidth,keepaspectratio]{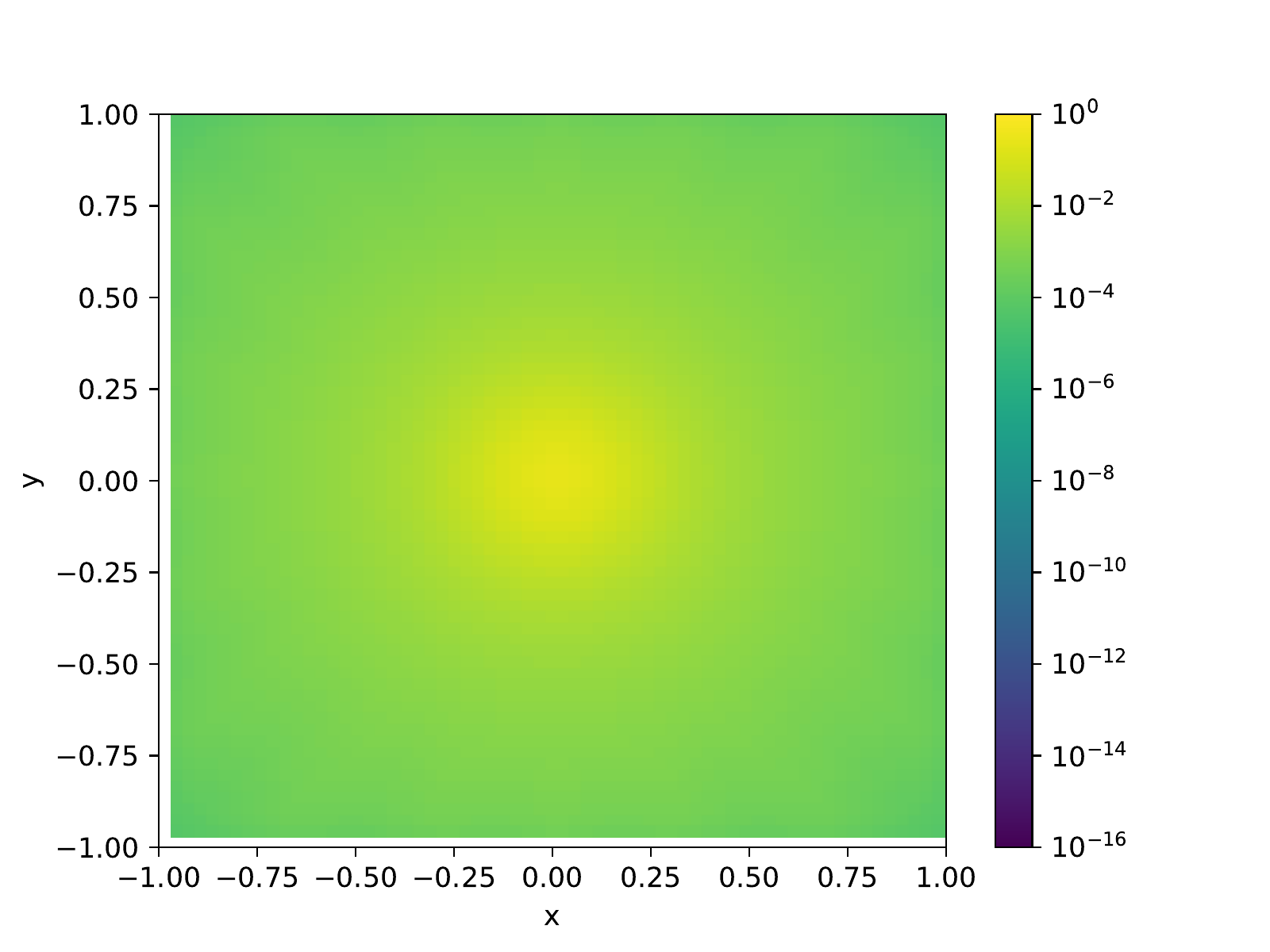}~
\includegraphics[width=0.33\textwidth,keepaspectratio]{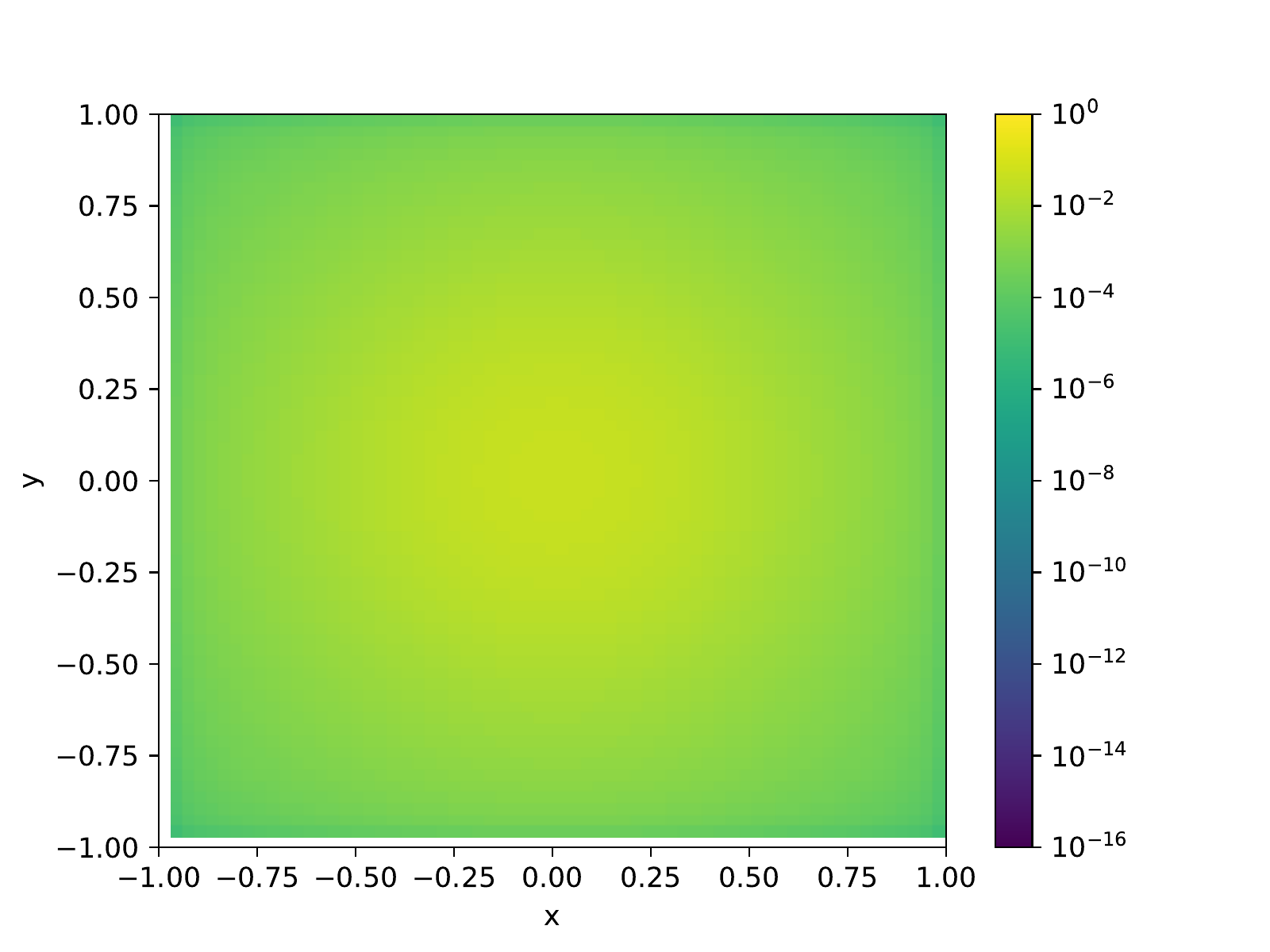}
\includegraphics[width=0.33\textwidth,keepaspectratio]{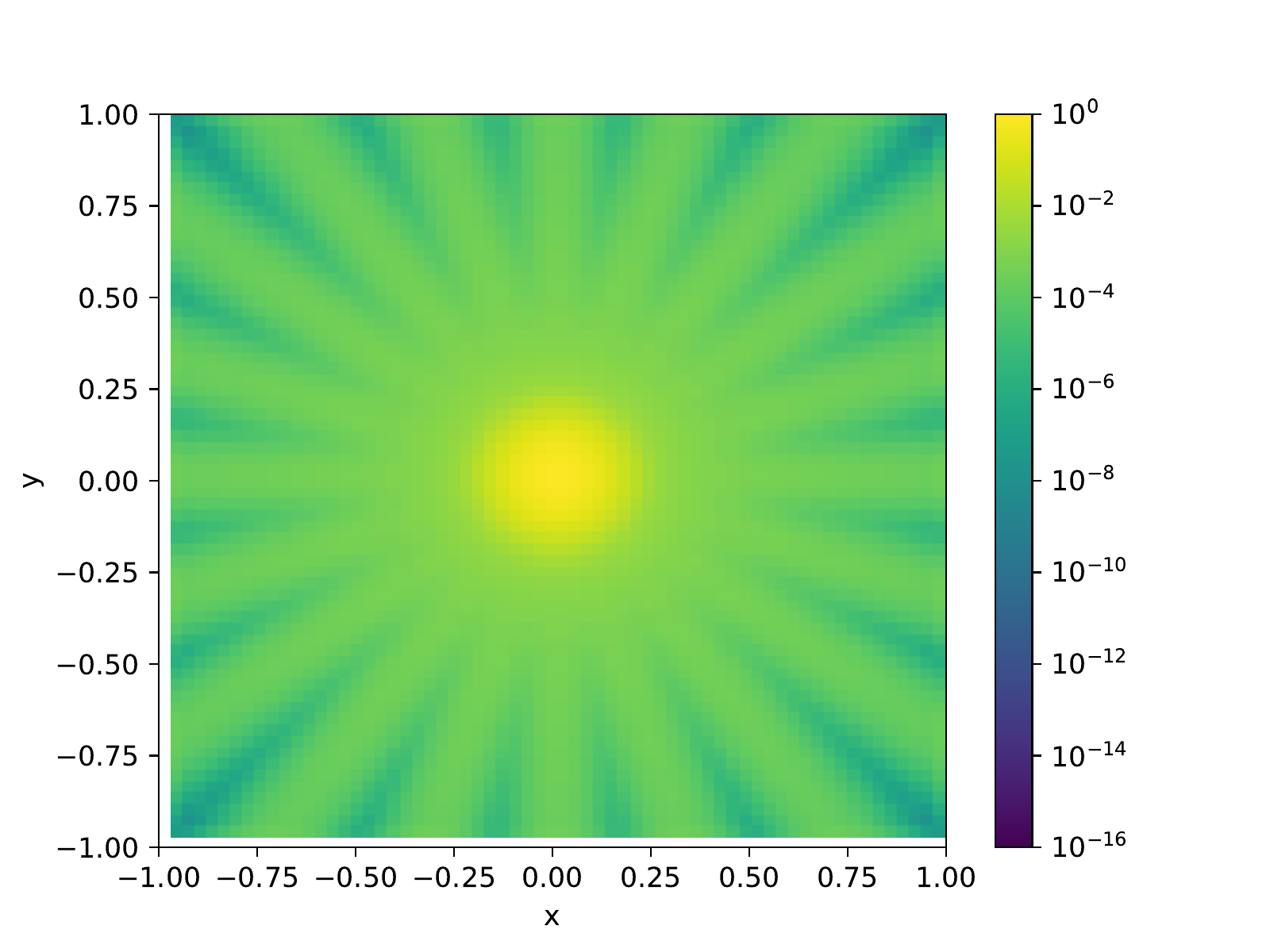}~
\includegraphics[width=0.33\textwidth,keepaspectratio]{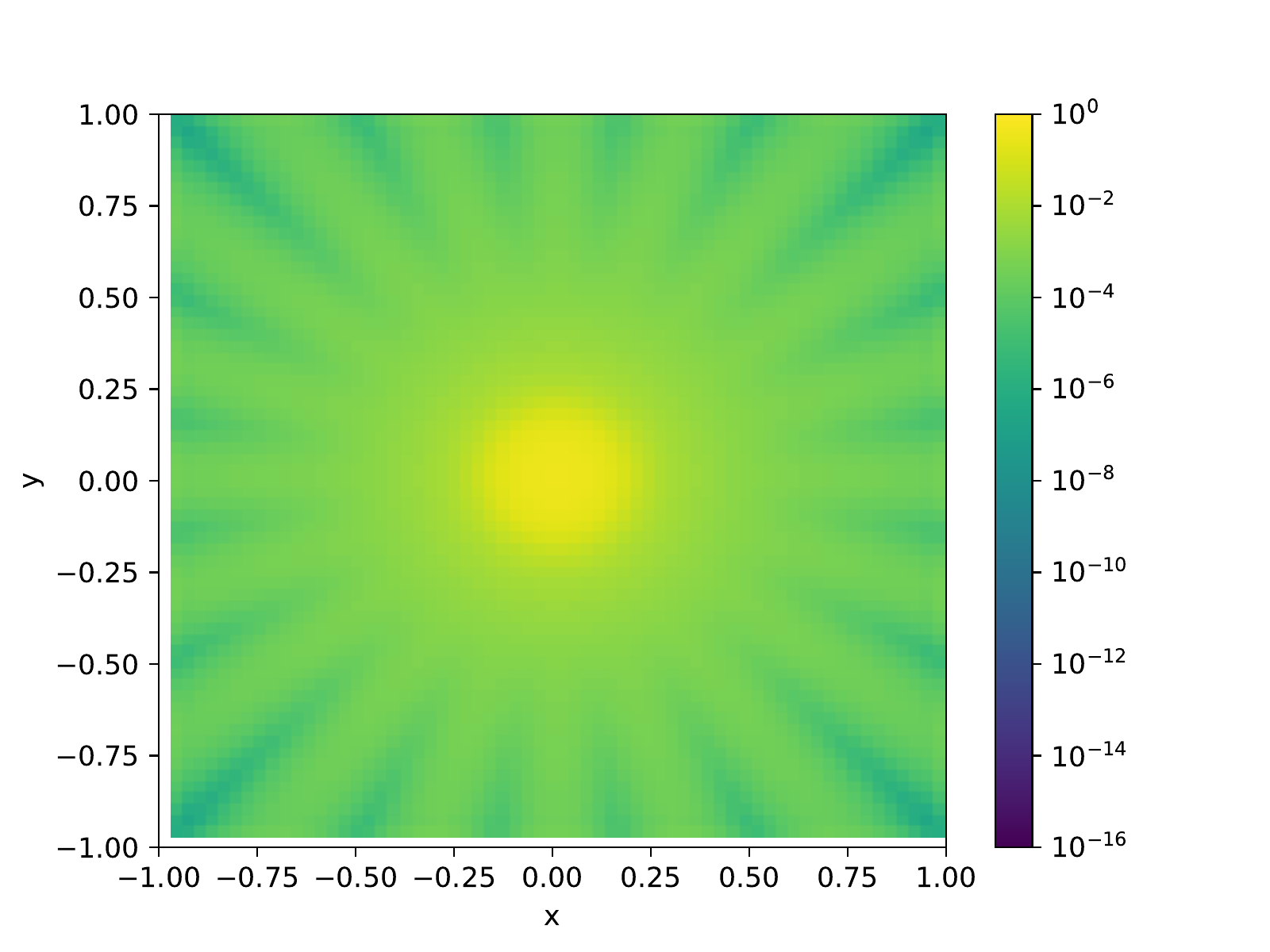}~
\includegraphics[width=0.33\textwidth,keepaspectratio]{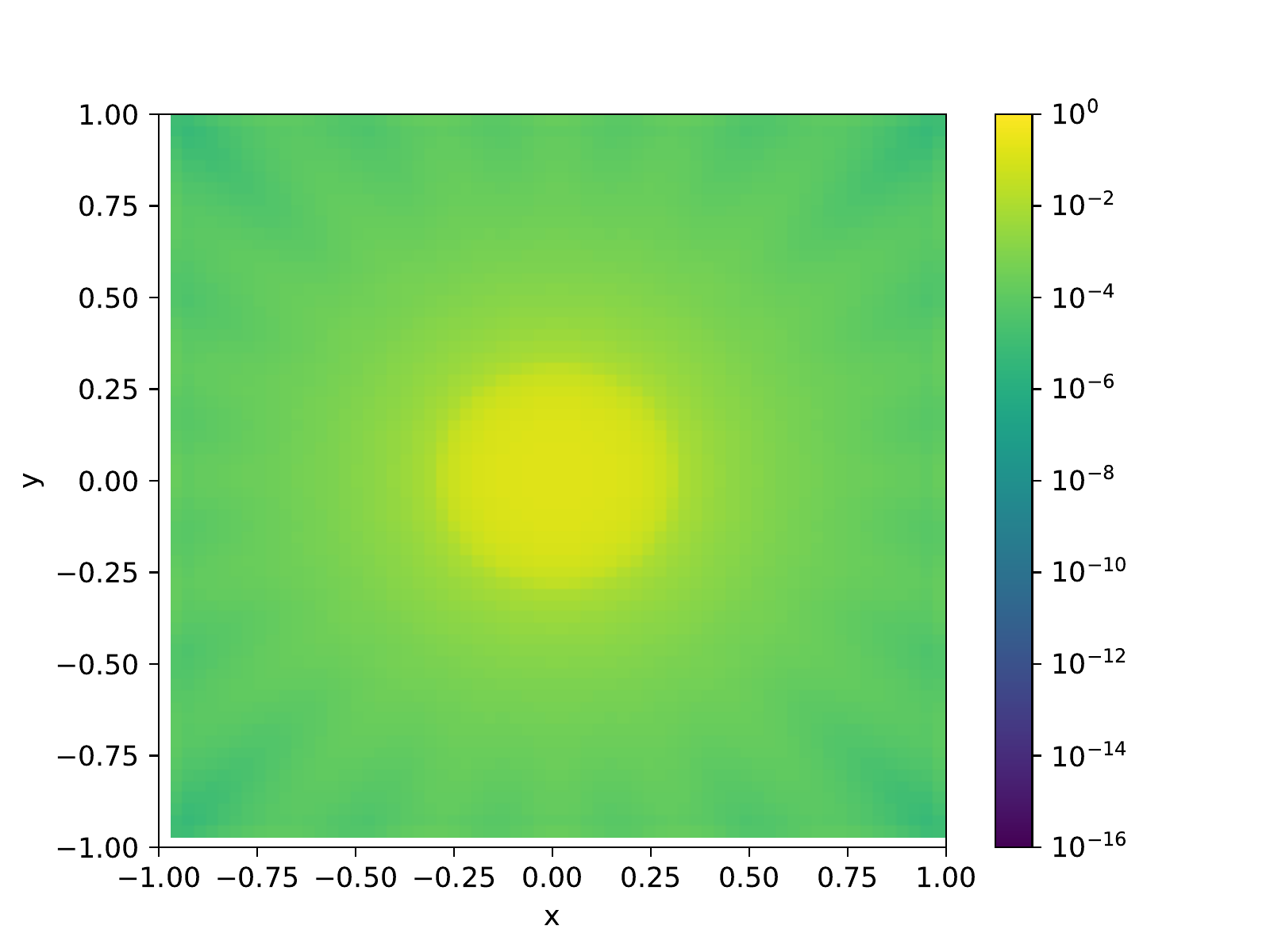}
\caption{Solution to the porous media equation at $t=0.1$. First row: $m=1$, second row: $m=2$; first column: $s=0.2$, second column: $s=0.5$, third column: $s=0.8$. Note the logarithmic scale for colors.}
\label{fig:final}
\end{figure}

\section{Conclusion}\label{sect:conc}

We have studied the NURBS-based isogeometric analysis for PDEs involving the fractional Laplacian operator, and applied the proposed algorithm to solve the fractional porous media equation. The proposed algorithm showed promising results on the benchmark problems where monotonous convergence was obtained and a consistent convergence rate was observed throughout multiple test cases. In the porous media equation case, the numerical solution coincided with the exact solution when the latter was available. 

We conclude that isogeometric analysis is a viable tool for nonlocal problems modeled by the fractional Laplacian. 
For domains admitting a Cartesian grid discretization, meaning that the domain can be tessellated with squares (cubes), the finite difference method is applicable and easy to implement. The main challenge for many applications is that the domains are complex. The finite element method~\cite{grossmann2012isogeometric} is a standard approach to tackle complex domain and it works with the weak formulation of the partial differential equations. However, the implementation is substantially nontrivial for the fractional Laplacian due to the singularity in the integrand. Working with the strong form is easier but requires higher smoothness assumptions on the solution representation. The NURBS-based basis functions in the isogeometric analysis are sufficiently smooth with appropriate degrees, thus we can work with the strong form directly and have simpler implementations. Moreover, the isogeometric analysis enables defining precisely many complex geometries (in this work, we worked with both the disk domain and the square domain) \cite{hassani2011application} and has higher accuracy per degree of freedom than the finite element analysis.

Despite the many strengths of our isogeometric collocation method, it requires pre\-com\-pu\-ting $\mathbf{L}$ in \Cref{alg:poisson}, which can be very expensive since each quadrature point corresponding to the collocation point must be located. That is, we need to find the coordinates in the parameter space from the coordinates of the quadrature points in the physical space, which may require solving many nonlinear equations. This is not a problem with finite difference methods (when applicable) or finite element methods since those methods do not parametrize the physical domain.

In contrast to integer-order differential operators, fractional operators lead to denser coefficient matrices, which have higher storage and computational cost. However, the fractional Laplacian is an important tool when designing high fidelity models. An important next step is to develop efficient algorithms and explore the potential of isogeometric analysis for fractional PDEs; for examples, the computation of integrals based on quadrature rules can be accelerated using fast multiple method~\cite{carrier1988fast,cheng1999fast,nishimura1999fast,nishimura2002fast,song1995multilevel,darve2000fast,coulier2017inverse,darve2011fast,fong2009black}.

\section{Acknowledgements}

Kailai Xu thanks the Stanford Graduate Fellowship in Science \& Engineering and the 2018 Schlumberger Innovation Fellowship for their financial support. Kailai Xu and Eric Darve are also supported in part by the Applied Mathematics Program within the Department of Energy (DOE) Office of Advanced Scientific Computing Research (ASCR), through the Collaboratory on Mathematics and Physics-Informed Learning Machines for Multiscale and Multiphysics Problems Research Center (DESC0019453).

\newpage

\appendix

The appendix includes the basic concepts of isogeometric analysis which are relevant for our implementation.

\section{Preliminary: Isogeometric Analysis}\label{sect:iga}
\subsection{B-spline}

We can describe B-splines in terms of a \textit{knot vector} in the \textit{parameter space}. A knot vector is specified by a non-decreasing set of coordinates 
\begin{equation}
    \mathcal{U} = \{u_1, u_2, \ldots, u_{l+1}\}
\end{equation}
where $u_i$ is called the \textit{knot}, and satisfies
\begin{equation}\label{equ:ueq}
    0 = u_1 = u_1 = \ldots = u_{p+1} \leq u_{p+2}\leq \ldots \leq u_{l-p} \leq u_{l-p+1} =u_{l-p+2} =\ldots = u_{l+1} = 1
\end{equation}
We allow the same value to occur multiple times and it will affect the continuity of the B-spline.  

The $i$-th B-spline basis function of $p$-degree $N_{i,p}$ can be defined recursively as 
\begin{align}
    {B_{i,0}}(u) &= \begin{cases}
        1 & \text{if $u_i \leq u < u_{i+1}$,} \\
        0 & \text{otherwise;}
    \end{cases}\\
    B_{i,p}(u) &= \frac{u-u_i}{u_{i+p}-u_i}B_{i,p-1}(u) + \frac{u_{i+p+1}-u}{u_{i+p+1}-u_{i+1}}B_{i+1,p-1}(u)
\end{align}
We note that B-splines have the following properties:
\begin{itemize}
\item The basis $B_{i,p}$ has compact support in $[u_i, u_{i+p+1}]$. 
\item For a knot vector of size $l+1$, there are $l-p$ B-spline basis functions in total. 
    \item They form a partition of unity, i.e.,
    \begin{equation}
        \sum_{i=1}^{l-p} B_{i,p} = 1
    \end{equation}
    \item Assume that all the inequalities in \Cref{equ:ueq} are strict, then $B_{i,p}\in\mathcal{C}^{p-1}$ but $B_{i,p}\not\in\mathcal{C}^{p}$.
\end{itemize}

\subsection{NURBS}

The NURBS basis function is created from B-splines by
\begin{equation}\label{equ:Nip}
    N_{i,p}(u) = \frac{w_iB_{i,p}(u)}{\sum_{j=1}^{l-p}w_jB_{j,p}(u)}
\end{equation}
where $w_i$ are $l-p$ weights assigned to each B-spline basis function. Bivariate NURBS are constituted by~(suppressing the degrees $p$ and $q$ for $u$ and $v$)
\begin{equation}
    N_{k,l}(u, v) = \frac{w_{kl}B_k(u)B_l(v)}{\sum_{i=1}^{l_u-p}\sum_{j=1}^{l_v-q}w_{ij}B_i(u)B_j(v) }
\end{equation}
where $l_u+1$ and $l_v+1$ are the number of knots for $u$ and $v$ knot vectors.

A domain in one-dimension and two-dimension Cartesian space can be constructed with 
\begin{equation}\label{equ:map}
  \begin{aligned}
    X &= {F}(u) = \sum_{i=1}^{l-p}N_{i,p}(u)\tilde X_{i}\\
    \bX &= \mathbf{F}(u,v) = \sum_{i=1}^{l_u-p}\sum_{j=1}^{l_v-q} N_{k,l}(u,v)\tilde\bX_{ij}
\end{aligned}
\end{equation}
where $\tilde X_{i}$ and $\tilde \bX_{ij}$ are called \textit{control points}. The equation \Cref{equ:map} defines a mapping from parameter space to the physical space. \Cref{fig:disk} shows two examples of 2D domains constructed from NURBS basis functions. The red dots are the control points while the blue patches are subdomains corresponding to $[u_i,u_{i+1}]\times [v_j, v_{j+1}]$ in the parameter space.
\begin{figure}[htpb] % \usepackage{float}
\centering
\includegraphics[width=0.8\textwidth,keepaspectratio]{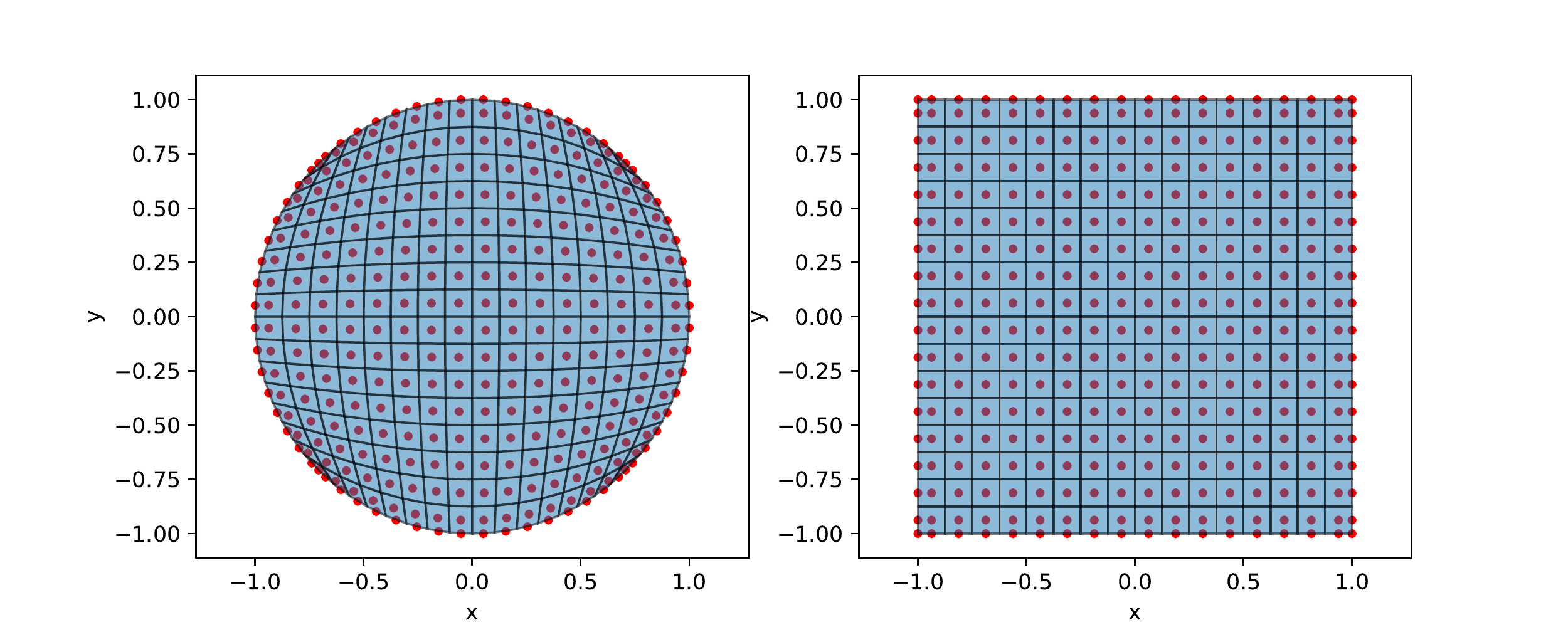}
\caption{Two examples of 2D domains constructed from NURBS basis functions. The red dots are the control points while the blue patches are subdomains corresponding to $[u_i,u_{i+1}]\times [v_j, v_{j+1}]$ in the parameter space.}
\label{fig:disk}
\end{figure}

In isogeometric analysis we use the same functions $N_{kl}$ as basis functions for approximating the solutions. The interpolation function space is defined by a ``push forward'' operator
\begin{equation}
    \mathbf{V} = \mathrm{span}\{N_{kl}\circ \mathbf{F}^{-1}\}_{k=1,2,\ldots,l_u-p;~l=1,2,\ldots,l_v-q}
\end{equation}
i.e., the numerical solution is given by
\begin{equation}
    u_h(x,y) = \sum_{i=1}^{l_u-p}\sum_{j=1}^{l_v-q} N_{ij}(\mathbf{F}^{-1}(x,y)) c_{ij}
\end{equation}
where $c_{ij}$ are coefficients determined by solving the governing equations.

\subsection{Knot Insertion}

The process of mesh refinement is critical to validate the numerical model. In isogeometric analysis, \textit{knot insertion} provides us a way to easily refine the NURBS mesh. For any NURBS curve in the first formula of \Cref{equ:map}, we can view it as the projection of a B-spline curve in the two-dimensional space, where the coordinates are given as 
\begin{equation}
    \bar \bX_i =[w_i\tilde X_i, ~w_i]
\end{equation}
Given the $l-p$ basis functions and the knot vector $\{u_1, u_2, \ldots, u_{l+1}\}$, and $\bar u\in [u_k, u_{k+1})$ be a desired new knot, the new $l+2$ control points $\{\bar \bX_1^\dag,  \bar \bX_2^\dag, \ldots, \bar \bX_{l+2}^\dag \}$ are formed from the original control points by
\begin{equation}
    \bar \bX_i^\dag = \alpha_i \bar \bX_i + (1-\alpha_i) \bar \bX_{i-1}
\end{equation}
where
\begin{equation}
    \alpha_i = \begin{cases}
        1 & 1\leq i \leq k-p,\\
        \frac{\bar u-u_i}{u_{i+p}-u_i} & k-p+1\leq i \leq k\\
        0 & k+1\leq i\leq l+2
    \end{cases}
\end{equation}
The new projected coordinates of the control points in the physical space are obtained using 
\begin{equation}
    \bar X_i^\dag = \frac{(\bar \bX_i^\dag)_1 }{(\bar \bX_i^\dag)_2}
\end{equation}
where $(\ldots)_j$ denotes the $j$-th component. Knot insertion for two or higher dimension NURBS domains can be done separately for each dimension.

\subsection{Computing the Coordinates in the Parameter Space}

In the course of computing the fractional Laplacian, we need to compute the coordinate in the parametric space given its physical space location. NURBS does not provide a direct way to do this and thus a nonlinear equation is usually required to be solved. Specifically, given $\bx = \begin{bmatrix}
    x\\
    y
\end{bmatrix}$, we want to find $(u,v)$ such that
\begin{equation}\label{equ:gn}
    \mathbf{F}(u,v) =\sum_{i=1}^{l_u-p}\sum_{j=1}^{l_v-q} N_{i,j}(u,v)\tilde\bX_{ij} = \begin{bmatrix}
        x\\
        y
    \end{bmatrix}
\end{equation}
We can adopt the Gauss-Newton method for solving \Cref{equ:gn}, where the gradient
\begin{equation}
  \nabla \mathbf{F}(u,v) = \begin{bmatrix}
    \frac{\partial F_1}{\partial u} & \frac{\partial F_1}{\partial v}\\
    \frac{\partial F_2}{\partial u} & \frac{\partial F_2}{\partial v} 
\end{bmatrix}
\end{equation}
is required. The gradient can be computed by noticing that
\begin{equation}
  \begin{aligned}
  S(u,v) &= {{{\left( {\sum\limits_{i,j} {{w_{ij}}{B_i}(u){B_j}(v)} } \right)}^2}}\\
    {\frac{\partial {N_{kl}}(u,v)}  {\partial u}} &= \frac{{w_{kl}}B_k^\prime (u){B_l}(u)\sum\limits_{i,j} {{w_{ij}}{B_i}(u){B_j}(v)} - {w_{kl}}{B_k}(u){B_l}(v)\sum\limits_{i,j} {{w_{ij}}B_i^\prime (u){B_j}(v)} } { S(u,v)}\\
    {\frac{\partial {N_{kl}}(u,v)}  {\partial v}} &= \frac{{w_{kl}}B_k^{}(u)B_l^\prime (u)\sum\limits_{i,j} {{w_{ij}}{B_i}(u){B_j}(v)} - {w_{kl}}{B_k}(u){B_l}(v)\sum\limits_{i,j} {{w_{ij}}B_i^{}(u)B_j^\prime (v)} } { S(u,v)}
\end{aligned}
\end{equation}
The algorithm for computing the coordinates in the parameter space is shown in \Cref{alg:gn}.

\begin{algorithm}[htpb]
\begin{algorithmic}[1]
\STATE \textbf{Input}: $\bx_0$ $\quad$ \textbf{Output}: $\bu$
\STATE Initialize $\bu\gets \begin{bmatrix}
    0.5\\
    0.5
\end{bmatrix}$, $\bx\gets \mathbf{F}(\bu)$
\WHILE{$|\bx_0-\bx|>$ tolerance}
\STATE $\mathbf{d}\gets \left(\nabla\mathbf{F}(\bu)\right)^{-1}(\bx-\bx_0)$
\STATE $\bu\gets \bu - \mathbf{d}$
\STATE $\bx\gets \mathbf{F}(\bu)$
\ENDWHILE
\RETURN $\bu$
\end{algorithmic}
\caption{Gauss-Newton Method for Computing the Coordinates in the Parameter Space.}
\label{alg:gn}
\end{algorithm}

\newpage

\bibliographystyle{plain}
\bibliography{IGA.bib}

\begin{thebibliography}{10}

\bibitem{slidespd48:online}
\_\_\_.
\newblock
  \url{https://www.uni-muenster.de/AMM/num/Vorlesungen/PDEI_SS14/material/slides.pdf}.
\newblock Lecture Notes.

\bibitem{acosta2017short}
Gabriel Acosta, Francisco~M Bersetche, and Juan~Pablo Borthagaray.
\newblock A short {FE} implementation for a 2d homogeneous dirichlet problem of
  a fractional {L}aplacian.
\newblock {\em Computers \& Mathematics with Applications}, 74(4):784--816,
  2017.

\bibitem{ainsworth2017aspects}
Mark Ainsworth and Christian Glusa.
\newblock Aspects of an adaptive finite element method for the fractional
  {L}aplacian: a priori and a posteriori error estimates, efficient
  implementation and multigrid solver.
\newblock {\em Computer Methods in Applied Mechanics and Engineering},
  327:4--35, 2017.

\bibitem{ainsworth2017towards}
Mark Ainsworth and Christian Glusa.
\newblock Towards an efficient finite element method for the integral
  fractional {L}aplacian on polygonal domains.
\newblock {\em arXiv preprint arXiv:1708.01923}, 2017.

\bibitem{auricchio2010isogeometric}
F~Auricchio, L~Beirao Da~Veiga, TJR Hughes, A\_ Reali, and G~Sangalli.
\newblock Isogeometric collocation methods.
\newblock {\em Mathematical Models and Methods in Applied Sciences},
  20(11):2075--2107, 2010.

\bibitem{bazilevs2006isogeometric}
Yuri Bazilevs, L~Beirao~da Veiga, J~Austin Cottrell, Thomas~JR Hughes, and
  Giancarlo Sangalli.
\newblock Isogeometric analysis: approximation, stability and error estimates
  for h-refined meshes.
\newblock {\em Mathematical Models and Methods in Applied Sciences},
  16(07):1031--1090, 2006.

\bibitem{carrier1988fast}
J~Carrier, Leslie Greengard, and Vladimir Rokhlin.
\newblock A fast adaptive multipole algorithm for particle simulations.
\newblock {\em SIAM journal on scientific and statistical computing},
  9(4):669--686, 1988.

\bibitem{cheng1999fast}
Hongwei Cheng, Leslie Greengard, and Vladimir Rokhlin.
\newblock A fast adaptive multipole algorithm in three dimensions.
\newblock {\em Journal of computational physics}, 155(2):468--498, 1999.

\bibitem{coulier2017inverse}
Pieter Coulier, Hadi Pouransari, and Eric Darve.
\newblock The inverse fast multipole method: using a fast approximate direct
  solver as a preconditioner for dense linear systems.
\newblock {\em SIAM Journal on Scientific Computing}, 39(3):A761--A796, 2017.

\bibitem{darve2000fast}
Eric Darve.
\newblock The fast multipole method: numerical implementation.
\newblock {\em Journal of Computational Physics}, 160(1):195--240, 2000.

\bibitem{darve2011fast}
Eric Darve, Cris Cecka, and Toru Takahashi.
\newblock The fast multipole method on parallel clusters, multicore processors,
  and graphics processing units.
\newblock {\em Comptes Rendus Mecanique}, 339(2-3):185--193, 2011.

\bibitem{de2010fractional}
Arturo de~Pablo, Fernando Quir{\'o}s, Ana Rodr{\'\i}guez, and Juan~Luis
  V{\'a}zquez.
\newblock A fractional porous medium equation.
\newblock {\em arXiv preprint arXiv:1001.2383}, 2010.

\bibitem{dyda2017eigenvalues}
Bart{\l}omiej Dyda, Alexey Kuznetsov, and Mateusz Kwa{\'s}nicki.
\newblock Eigenvalues of the fractional laplace operator in the unit ball.
\newblock {\em Journal of the London Mathematical Society}, 95(2):500--518,
  2017.

\bibitem{fong2009black}
William Fong and Eric Darve.
\newblock The black-box fast multipole method.
\newblock {\em Journal of Computational Physics}, 228(23):8712--8725, 2009.

\bibitem{gatto2015numerical}
Paolo Gatto and Jan~S Hesthaven.
\newblock Numerical approximation of the fractional {L}aplacian via $hp$-finite
  elements, with an application to image denoising.
\newblock {\em Journal of Scientific Computing}, 65(1):249--270, 2015.

\bibitem{gibou2005fourth}
Fr{\'e}d{\'e}ric Gibou and Ronald Fedkiw.
\newblock A fourth order accurate discretization for the laplace and heat
  equations on arbitrary domains, with applications to the stefan problem.
\newblock {\em Journal of Computational Physics}, 202(2):577--601, 2005.

\bibitem{grossmann2012isogeometric}
David Gro{\ss}mann, Bert J{\"u}ttler, Helena Schlusnus, Johannes Barner, and
  Anh-Vu Vuong.
\newblock Isogeometric simulation of turbine blades for aircraft engines.
\newblock {\em Computer Aided Geometric Design}, 29(7):519--531, 2012.

\bibitem{gulian2019machine}
Mamikon Gulian, Maziar Raissi, Paris Perdikaris, and George Karniadakis.
\newblock Machine learning of space-fractional differential equations.
\newblock {\em SIAM Journal on Scientific Computing}, 41(4):A2485--A2509, 2019.

\bibitem{hassani2011application}
Behrooz Hassani, S~Mehdi Tavakkoli, and NZ~Moghadam.
\newblock Application of isogeometric analysis in structural shape
  optimization.
\newblock {\em Scientia Iranica}, 18(4):846--852, 2011.

\bibitem{heydari2013two}
MH~Heydari, MR~Hooshmandasl, FM~Maalek Ghaini, and F~Fereidouni.
\newblock Two-dimensional {L}egendre wavelets for solving fractional poisson
  equation with dirichlet boundary conditions.
\newblock {\em Engineering Analysis with Boundary Elements}, 37(11):1331--1338,
  2013.

\bibitem{heydari2018legendre}
Mohammad~Hossein Heydari and Zakieh Avazzadeh.
\newblock {L}egendre wavelets optimization method for variable-order fractional
  poisson equation.
\newblock {\em Chaos, Solitons \& Fractals}, 112:180--190, 2018.

\bibitem{heydari2019wavelet}
Mohammad~Hossein Heydari, Zakieh Avazzadeh, and Malih~Farzi Haromi.
\newblock A wavelet approach for solving multi-term variable-order time
  fractional diffusion-wave equation.
\newblock {\em Applied Mathematics and Computation}, 341:215--228, 2019.

\bibitem{heydari2019computational}
Mohammad~Hossein Heydari, Zakieh Avazzadeh, and Yin Yang.
\newblock A computational method for solving variable-order fractional
  nonlinear diffusion-wave equation.
\newblock {\em Applied Mathematics and Computation}, 352:235--248, 2019.

\bibitem{hooshmandasl2016numerical}
MR~Hooshmandasl, MH~Heydari, and C~Cattani.
\newblock Numerical solution of fractional sub-diffusion and time-fractional
  diffusion-wave equations via fractional-order {L}egendre functions.
\newblock {\em The European Physical Journal Plus}, 131(8):268, 2016.

\bibitem{huang2014numerical}
Yanghong Huang and Adam Oberman.
\newblock Numerical methods for the fractional {L}aplacian: A finite
  difference-quadrature approach.
\newblock {\em SIAM Journal on Numerical Analysis}, 52(6):3056--3084, 2014.

\bibitem{hughes2005isogeometric}
Thomas~JR Hughes, John~A Cottrell, and Yuri Bazilevs.
\newblock Isogeometric analysis: {CAD}, finite elements, {NURBS}, exact
  geometry and mesh refinement.
\newblock {\em Computer methods in applied mechanics and engineering},
  194(39-41):4135--4195, 2005.

\bibitem{johnson2005higher}
Richard~W Johnson.
\newblock Higher order b-spline collocation at the greville abscissae.
\newblock {\em Applied Numerical Mathematics}, 52(1):63--75, 2005.

\bibitem{kwasnicki2017ten}
Mateusz Kwa{\'s}nicki.
\newblock Ten equivalent definitions of the fractional laplace operator.
\newblock {\em Fractional Calculus and Applied Analysis}, 20(1):7--51, 2017.

\bibitem{kyprianou2017unbiased}
Andreas~E Kyprianou, Ana Osojnik, and Tony Shardlow.
\newblock Unbiased `walk-on-spheres' monte carlo methods for the fractional
  {L}aplacian.
\newblock {\em IMA Journal of Numerical Analysis}, 38(3):1550--1578, 2017.

\bibitem{lischke2018fractional}
Anna Lischke, Guofei Pang, Mamikon Gulian, Fangying Song, Christian Glusa,
  Xiaoning Zheng, Zhiping Mao, Wei Cai, Mark~M Meerschaert, Mark Ainsworth,
  et~al.
\newblock What is the fractional {L}aplacian?
\newblock {\em arXiv preprint arXiv:1801.09767}, 2018.

\bibitem{maryshev2013adjoint}
Boris Maryshev, Alain Cartalade, Christelle Latrille, Maminirina Joelson, and
  Marie-Christine N{\'e}el.
\newblock Adjoint state method for fractional diffusion: parameter
  identification.
\newblock {\em Computers \& Mathematics with Applications}, 66(5):630--638,
  2013.

\bibitem{2018arXiv180203770M}
V.~{Minden} and L.~{Ying}.
\newblock A simple solver for the fractional {L}aplacian in multiple
  dimensions.
\newblock {\em ArXiv e-prints}, February 2018.

\bibitem{nishimura2002fast}
Naoshi Nishimura.
\newblock Fast multipole accelerated boundary integral equation methods.
\newblock {\em Applied mechanics reviews}, 55(4):299--324, 2002.

\bibitem{nishimura1999fast}
Naoshi Nishimura, Ken-ichi Yoshida, and Shoichi Kobayashi.
\newblock A fast multipole boundary integral equation method for crack problems
  in 3d.
\newblock {\em Engineering Analysis with Boundary Elements}, 23(1):97--105,
  1999.

\bibitem{plociniczak2014approximation}
{\L}ukasz P{\l}ociniczak.
\newblock Approximation of the {E}rd\'elyi--{K}ober operator with application
  to the time-fractional porous medium equation.
\newblock {\em SIAM Journal on Applied Mathematics}, 74(4):1219--1237, 2014.

\bibitem{reali2015introduction}
Alessandro Reali and Thomas~JR Hughes.
\newblock An introduction to isogeometric collocation methods.
\newblock In {\em Isogeometric Methods for Numerical Simulation}, pages
  173--204. Springer, 2015.

\bibitem{ros2014dirichlet}
Xavier Ros-Oton and Joaquim Serra.
\newblock The dirichlet problem for the fractional {L}aplacian: regularity up
  to the boundary.
\newblock {\em Journal de Math{\'e}matiques Pures et Appliqu{\'e}es},
  101(3):275--302, 2014.

\bibitem{song1995multilevel}
JM~Song and Weng~Cho Chew.
\newblock Multilevel fast-multipole algorithm for solving combined field
  integral equations of electromagnetic scattering.
\newblock {\em Microwave and Optical Technology Letters}, 10(1):14--19, 1995.

\bibitem{temizer2011contact}
I~Temizer, P~Wriggers, and TJR Hughes.
\newblock Contact treatment in isogeometric analysis with nurbs.
\newblock {\em Computer Methods in Applied Mechanics and Engineering},
  200(9-12):1100--1112, 2011.

\bibitem{wang2017adaptive}
Yufeng Wang, Hui Zhou, Hanming Chen, and Yangkang Chen.
\newblock Adaptive stabilization for {Q}-compensated reverse time migration.
\newblock {\em Geophysics}, 83(1):S15--S32, 2017.

\bibitem{xu2018calibrating}
Kailai Xu and Eric Darve.
\newblock Calibrating multivariate {L}{\'e}vy processes with neural networks.
\newblock {\em arXiv preprint arXiv:1812.08883}, 2018.

\bibitem{xu2018spectral}
Kailai Xu and Eric Darve.
\newblock Spectral method for the fractional {L}aplacian in 2d and 3d.
\newblock {\em arXiv preprint arXiv:1812.08325}, 2018.

\bibitem{xu2018radial}
Yiran Xu, Jingye Li, Guofei Pang, Zhikai Wang, and Xiaohong Chen.
\newblock Radial basis function collocation method for decoupled fractional
  {L}aplacian wave equations.
\newblock {\em arXiv preprint arXiv:1801.01206}, 2018.

\end{thebibliography}

\end{document}